\documentclass[twoside,english]{article}

\usepackage[T1]{fontenc}
\usepackage[latin9]{inputenc}
\usepackage[dvipdfm]{geometry}
\usepackage{prettyref}
\usepackage{units}
\usepackage{amsthm}
\usepackage{amsmath}
\usepackage{amssymb}
\usepackage[dvipdfm]{graphicx}

\makeatletter
\theoremstyle{plain}
\newtheorem{thm}{\protect\theoremname}
\theoremstyle{plain}
\newtheorem{lem}[thm]{\protect\lemmaname}
\theoremstyle{plain}
\newtheorem{cor}[thm]{\protect\corollaryname}
\theoremstyle{definition}
\newtheorem{example}[thm]{\protect\examplename}
\ifx\proof\undefined\
\newenvironment{proof}[1][\protect\proofname]{\par
\normalfont\topsep6\p@\@plus6\p@\relax
\trivlist
\itemindent\parindent
\item[\hskip\labelsep
\scshape
#1]\ignorespaces
}{%
\endtrivlist\@endpefalse
}
\fi
\theoremstyle{definition}
\newtheorem{defn}[thm]{\protect\definitionname}

\usepackage{hyperref}
  \hypersetup{colorlinks=true,citecolor=blue}
\date{}

\usepackage{fancyhdr}
\makeatother

\usepackage{babel}
\providecommand{\corollaryname}{Corollary}
\providecommand{\definitionname}{Definition}
\providecommand{\examplename}{Example}
\providecommand{\lemmaname}{Lemma}
\providecommand{\proofname}{Proof}
\providecommand{\theoremname}{Theorem}

\begin{document}
\global\long\def\FF{\mathbb{F}}
\global\long\def\KK{\mathbb{K}}
\global\long\def\NN{\mathbb{N}}
\global\long\def\RR{\mathbb{R}}
\global\long\def\CC{\mathbb{C}}
\global\long\def\ZZ{\mathbb{Z}}

\title{Graded Embeddings of Finite Dimensional Simple Graded Algebras}

\author{Ofir David}

\maketitle
\begin{center}
\textit{\small Department of Mathematics, Technion-Israel Institute
of Technology, Haifa 32000, Israel}\\
\textit{\small E-mail address: }\texttt{\textit{\small ofirdav@tx.technion.ac.il}}
\par\end{center}{\small \par}
\begin{abstract}
Let $A,B$ be finite dimensional $G$-graded algebras over an algebraically
closed field $\KK$ with $char(\KK)=0$, where $G$ is an abelian
group, and let $Id_{G}(A)$ be the set of graded identities of $A$
(res. $Id_{G}(B)$). We show that if $A,B$ are $G$-simple then there
is a graded embedding $\phi:A\rightarrow B$ iff $Id_{G}(B)\subseteq Id_{G}(A)$.
We also give a weaker generalization for the case where $A$ is $G$-semisimple
and $B$ is arbitrary.
\end{abstract}

\section{Introduction}

Let $A$ be an algebra and $G$ be a group. A $G$-grading of $A$
is a decomposition $A=\bigoplus_{g\in G}A_{g}$ where $A_{g}A_{h}\subseteq A_{gh}$.
A polynomial $p(x_{g_{1},1},...,x_{g_{n},n})$ with non-commutative
graded indeterminates is called a \textit{\emph{graded identity}}
of $A$ if $p(a_{1},...,a_{n})=0$ for any assignment in $A$ such
that $a_{i}\in A_{g_{i}}$, and we define $Id_{G}(A)$ to be the set
of the graded identities of $A$. 

If $A$ and $B$ are isomorphic as $G$-graded algebras then obviously
$Id_{G}(A)=Id_{G}(B)$, and we can ask if the converse is also true.
In the current setting this is not true. For example, for any algebra
$A$ we have $Id_{G}(A)=Id_{G}(A\oplus A)$ while in general $A$
isn't isomorphic to $A\oplus A$. Another less trivial but important
example comes from scalar extension. It is well-known that if $A,B$
are algebras over a field $\FF$ with $char(\FF)=0$, $\KK$ is a
field extension of $\FF$ and $A\otimes_{\FF}\KK\cong_{G}B\otimes_{\FF}\KK$,
then $Id_{G}(A)=Id_{G}(B)$, even though $A$ doesn't have to be isomorphic
to $B$. For example the quaternion algebra $\mathbb{H}_{\RR}$ and
the matrix algebra $M_{2}(\RR)$ are not isomorphic, while $\mathbb{H}_{\RR}\otimes_{\RR}\CC\cong M_{2}(\CC)\cong M_{2}(\RR)\otimes_{\RR}\CC$
so $Id(\mathbb{H}_{\RR})=Id(M_{2}(\RR))$.

To avoid such counter examples as above, we restrict ourselves to
algebraically closed fields, and also assume that $A$ and $B$ are
finite dimensional $G$-simple algebras. In this case, it is true
that if $Id_{G}(A)=Id_{G}(B)$ then $A\cong_{G}B$. This was first
proved for $G$ abelian and the order of any finite subgroup of $G$
is invertible in $\KK$ (Koshlukov and Zaicev \cite{koshlukov_identities_2010})
and was later proved for $G$ arbitrary and $char(\KK)=0$ (Aljadeff
and Haile \cite{aljadeff_simple_2011}).

These two results show that in order to identify a finite dimensional
$G$-simple algebra, all we need to know is its graded identities.
Seeing that the equality $Id_{G}(A)=Id_{G}(B)$ means that $A\cong_{G}B$
in these cases, we can ask what other relations between two $G$-graded
algebras $A$ and $B$ can be derived from relations between their
respective ideals of identities.

If $A\subseteq B$ as graded algebras, then any polynomial identity
of $B$ is certainly an identity of $A$ so $Id_{G}(B)\subseteq Id_{G}(A)$,
so here too we can ask whether or not the converse is true. Suppose
now that we are in the non-graded case, and we have $Id(B)\subseteq Id(A)$.
If $A,B$ are finite dimensional simple algebras over an algebraically
closed field $\KK$, then by Artin-Wedderburn theorem we know that
$A=M_{n}(\KK)$ , $B=M_{m}(\KK)$ for some $m,n\in\NN$. Let $St_{r}(x_{1},...,x_{r})={\displaystyle \sum_{\sigma\in S_{r}}}(-1)^{\sigma}\prod x_{\sigma(i)}$
be the standard polynomial with $r$ indeterminates, then by Amitsur-Levitzki
theorem $St_{r}\in Id(M_{n}(\KK))$ iff $r\geq2n$. Using the Amitsur-Levitzki
theorem we first get that $St_{2m}\in Id(B)$ so $St_{2m}\in Id(A)$,
and then using it a second time we get that $2m\geq2n\;\Rightarrow\; m\geq n$.
Since $m\geq n$ we can find an injective map $\varphi:M_{n}(\KK)\rightarrow M_{m}(\KK)$,
so in the non-graded case, for two simple finite dimensional algebras
$A,B$ there is an embedding $\varphi:A\rightarrow B$ iff there is
an inclusion $Id(B)\subseteq Id(A)$. 

The matrix algebras $M_{n}(\KK)$, as in the example above, together
with $M_{n}(E)$ and $M_{n,m}(E)$, where $E$ is the infinite Grassmann
algebra, play a very important part in the classical PI theory of
non-graded algebras. For a $T$ ideal $I$ of identities and an algebra
$A$ we denote by $I(A)$ the ideal of all the evaluations of polynomials
in $I$ with elements in $A$. In \cite{kemer_varieties_1985} Kemer
showed that for any PI algebra $A$ of characteristic zero there is
a $T$ ideal $I$ such that $\nicefrac{A}{I(A)}$ is PI equivalent
to a direct sum of verbally prime algebras, and every verbally prime
algebra is PI equivalent to one of $M_{n}(\KK),\; M_{n}(E)$ or $M_{n,m}(E)$.
The question, if $Id(A)\subseteq Id(B)$ implies $B\subseteq A$,
for these three types of algebras, was answered positively by Berele
in \cite{berele_classification_1993}, with a slight exception in
the case of $M_{n,m}(E)$.

As we showed above, the proof for non-graded finite dimensional simple
algebras, or in other words, matrix algebras, is fairly simple. In
the graded case the situation is much more complicated. Each matrix
algebra $M_{n}(\KK)$ may have many non-isomorphic $G$-gradings,
and so in particular they are $G$-simple, and there are also other
graded algebras which are not isomorphic (non-graded) to matrix algebras,
but still $G$-simple. In this paper, we show that the claim is true
for $G$-simple finite dimensional algebras if the group $G$ is abelian. 
\begin{thm}
Let $G$ be an abelian group, and $\KK$ an algebraically closed field
with $char(\KK)=0$. Let $A,B$ be two $G$-graded, finite dimensional
algebras over $\KK$.
\begin{enumerate}
\item If $A,B$ are $G$-simple then there is a graded embedding $A\hookrightarrow_{G}B$
iff $Id_{G}(B)\subseteq Id_{G}(A)$.
\item If $A$ is $G$-semisimple, $G$ is finite and $B$ has a unit then
$Id_{G}(B)\subseteq Id_{G}(A)$ iff there is a graded embedding $A\hookrightarrow_{G}B^{N}$
for $N$ large enough. 
\item If $A$ is simple then we can choose $N=1$, so $Id_{G}(B)\subseteq Id_{G}(A)$
iff there is a graded embedding $A\hookrightarrow_{G}B$.
\end{enumerate}
\end{thm}
We also give a special case for $G$ arbitrary. Let $\left(g_{1},...,g_{n}\right)\in G^{n}$
be a tuple of elements in $G$, then we can define a $G$-grading
on $M_{n}(\KK)$ by setting $\deg(E_{i,j})=g_{i}^{-1}g_{j}$ where
$E_{i,j}$ is the zero matrix with $1$ in the $(i,j)$ location.
We call such a grading an elementary grading.
\begin{thm}
Let $G$ be an arbitrary group and $\KK$ an algebraically closed
field with $char(\KK)=0$. Let $A$ be a finite dimensional $G$-simple
algebra, and $B=M_{n}(\KK)$ with an elementary $G$-grading, then
$Id_{G}(B)\subseteq Id_{G}(A)$ iff there is a graded embedding $A\hookrightarrow_{G}B$.
\end{thm}
Let $\KK$ be any field with $char(\KK)=0$ and $\overline{\KK}$
its algebraic closure. As we mentioned before, a trivial reason for
two graded algebras $A,B$ over $\KK$ to have the same graded identities
is that $A\otimes_{\KK}\overline{\KK}\cong_{G}B\otimes_{\KK}\overline{\KK}$,
which means that $B$ is a twisted form of $A$ in the language of
Galois descent. Let $A$ be an algebra over $\KK$ and suppose that
$\bar{A}=A\otimes_{\KK}\bar{\KK}$ is $G$-simple over $\overline{\KK}$
and we want to build a generic algebra $A_{gen}$ that specialize
exactly to all the twisted forms of $A$. If $R_{1}$ is a graded
algebra and $R_{2}$ is a homomorphic image of $R_{1}$ then clearly
$Id_{G}(R_{1})\subseteq Id_{G}(R_{2})$, so in our case we will have
$Id_{G}(A_{gen})\subseteq Id_{G}(B)$ for any homomorphic image $B$
of $A_{gen}$, and in particular $Id_{G}(A_{gen})\subseteq Id_{G}(A)$. 

With this in mind, let $\KK\left\langle X_{G}\right\rangle $ be the
free algebra with noncommutative indeterminates $\left\{ x_{g,i}\right\} _{g\in G,i\in\NN}$
and let $\tilde{A}=\nicefrac{\KK\left\langle X_{G}\right\rangle }{Id_{G}(A)}$,
then in particular we have $Id_{G}(A)=Id_{G}(\tilde{A})$. It is clear
that any finitely generated algebra $B$ with $Id_{G}(A)\subseteq Id_{G}(B)$
is a homomorphic image of $\tilde{A}$, and in particular any twisted
form of $A$. The problem is that $\tilde{A}$ can have other homomorphic
images that are not twisted forms (for example the trivial algebra
$\left\{ 0\right\} $). Let $B$ be a specialization of $\tilde{A}$,
set $\bar{B}=B\otimes_{\KK}\bar{\KK}$ and assume that $\bar{B}$
is finite dimensional simple, then the theorem above shows that though
we don't necessarily have $\bar{B}\cong_{G}\overline{A}$, we do know
that $\bar{B}\hookrightarrow_{G}\overline{A}$. This shows that a
good place to start to look for $A_{gen}$ is with the relative free
algebra over $Id_{G}(A)$. Such constructions were already carried
for twisted group algebras in \cite{aljadeff_graded_2010}(definition
inside the paper).\\

In the second part of theorem 1, we see that if $A$ and $B$ are
$G$-semisimple, then $Id_{G}(B)\subseteq Id_{G}(A)$ doesn't have
to come from graded embedding $A\hookrightarrow_{G}B$. The situation
becomes much simpler when we consider only equality of the ideals,
and combine it with the results from \cite{koshlukov_identities_2010,aljadeff_simple_2011}
about isomorphisms of $G$-simple algebras. 

The next lemma was proved by Berele in \cite{berele_classification_1993}
for the case of verbally prime algebras, and we give here an easy
proof for the case of $G$-simple algebras. We say that a set $\left\{ A_{i}\right\} _{1}^{n}$
of graded algebras is minimal if for all $j$ we have $\bigcap_{k}Id_{G}(A_{k})\varsubsetneqq\bigcap_{k\neq j}Id_{G}(A_{k})$. 
\begin{lem}
Let $\KK$ be an arbitrary field. Let $\left\{ A_{i}\right\} _{1}^{n},\left\{ B_{i}\right\} _{1}^{m}$
be two minimal sets of $G$-simple algebras. If $A=\bigoplus_{1}^{n}A_{i}$
, $B=\bigoplus_{1}^{m}B_{j}$ and $Id_{G}(A)=Id_{G}(B)$ then $m=n$,
and there is a permutation $\tau:\left[n\right]\rightarrow\left[n\right]$
such that $Id_{G}(A_{i})=Id_{G}(B_{\tau(i)})$.
\end{lem}
If we also know that $Id_{G}(A_{i})=Id_{G}(B_{j})$ implies that $A_{i}\cong_{G}B_{j}$
then the lemma above shows that $Id_{G}(A)=Id_{G}(B)$ implies that
$A\cong_{G}B$. We now use the results from\cite{koshlukov_identities_2010,aljadeff_simple_2011}
to conclude that:
\begin{cor}
Let $\left\{ A_{i}\right\} _{1}^{n},\left\{ B_{j}\right\} _{1}^{m}$
be minimal sets of $G$-simple finite dimensional algebras and let
$A=\bigoplus_{1}^{n}A_{i}$, $B=\bigoplus_{1}^{m}B_{j}$. Suppose
that
\begin{enumerate}
\item $G$ is abelian and for every finite subgroup $H\leq G$ we have $\left|H\right|^{-1}\in\KK$
\item or $G$ is arbitrary and $char(\KK)=0$
\end{enumerate}

then $Id_{G}(A)=Id_{G}(B)$ iff $A\cong_{G}B$.

\end{cor}

\newpage{}

\section{Preliminaries}

\subsection{Graded Algebras}

We start with some basic definitions.

Let $A$ be an algebra over a field $\KK$, and let $G$ be a group.
A \emph{$G$-grading} of $A$ is a decomposition of $A$ as a vector
space over $\KK$ to $\bigoplus_{g\in G}A_{g}$, such that $A_{g}\cdot A_{h}\subseteq A_{gh}$
for every $g,h\in G$. A sub-algebra (respectively left, right or
two sided ideal) $B\subseteq A$ is called a \textit{graded sub-algebra}
if $B={\displaystyle \bigoplus_{g\in G}}(B\cap A_{g})$. The algebra
$A$ is called \emph{$G$-simple} if it has no non-trivial two sided
graded ideals and $A\cdot A\neq0$.
\begin{example}
[Twisted group algebra]Let $G$ be a finite group and $\KK$ a field.
A function $\alpha:G\times G\rightarrow\KK^{\times}$ is said to be
a 2-cocycle of $G$ (denoted by $\alpha\in Z^{2}(G,\KK^{\times})$)
if it satisfies
\[
\forall u,v,w\in G\qquad\alpha(u,v)\alpha(uv,w)=\alpha(u,vw)\alpha(v,w)
\]
Let $A=\KK^{\alpha}G$ be the algebra such that as a vector space
it is the direct sum ${\displaystyle \bigoplus_{g\in G}}\KK U_{g}$
and the product is defined by $U_{g}U_{h}=\alpha(g,h)U_{gh}$. The
property of $\alpha$ above shows that $\KK^{\alpha}G$ is an associative
algebra, and setting $A_{g}=\KK U_{g}$ induces a grading on $\KK^{\alpha}G$.
Taking a new basis $V_{g}=\alpha(e,e)^{-1}U_{g}$ we get
\begin{eqnarray*}
\alpha(g,e)\alpha(ge,e) & = & \alpha(g,ee)\alpha(e,e)\quad\Rightarrow\quad\alpha(g,e)=\alpha(e,e)\\
V_{g}V_{e} & = & \alpha(e,e)^{-2}U_{g}U_{e}=\alpha(e,e)^{-2}\alpha(g,e)U_{g}=\alpha(e,e)^{-1}\alpha(g,e)V_{g}=V_{g}
\end{eqnarray*}
and in a similar way $V_{e}V_{g}=V_{g}$ so $V_{e}$ is a unit. Because
$0\neq V_{g}V_{g^{-1}}\in\KK V_{e}$, each graded component $A_{g}$
is spanned by an invertible element, and therefore each nonzero graded
ideal of $A$ must be all $A$, or in other words $\KK^{\alpha}G$
is $G$-simple.

For a tuple $\bar{g}\in G^{n}$ we denote by $\alpha(\bar{g})=\alpha(g_{1},...,g_{n})\in\KK$
the constant such that 
\[
U_{g_{1}}U_{g_{2}}\cdots U_{g_{n}}=\alpha(g_{1},...,g_{n})U_{g_{1}g_{2}\cdots g_{n}}
\]

\end{example}

\begin{example}
[Elementary Grading]Let $\bar{s}\in G^{n}$. We denote by $M_{\bar{s}}(\KK)$
the algebra $A=M_{n}(\KK)$ with the grading 
\[
A_{g}=span\left\{ E_{i,j}\;\mid\;\deg(E_{i,j})=s_{i}^{-1}s_{j}=g\right\} .
\]
$A$ is simple as an algebra so it must also be simple as a graded
algebra.
\end{example}
Let $A,B$ be two $G$-graded algebras. An algebra homomorphism $\varphi:A\rightarrow B$
is called graded if $\varphi(A_{g})\subseteq B_{g}$ for any $g\in G$,
and if this is true then $\ker(\varphi)$ is a graded two sided ideal.
Notice that if $A$ is simple then $\varphi$ is either the zero homomorphism,
or injective.

A theorem by Bahturin, Sehgal and Zaicev \cite{bahturin_finite-dimensional_2008}
gives us the structure of all $G$-simple algebras of finite dimension.
\begin{thm}
Let $R={\displaystyle \bigoplus_{g\in G}}R_{g}$ be a $G$-graded
finite-dimensional algebra over an algebraically closed field $\KK.$
Suppose that $char(\KK)=0$ or $char(\KK)$ is coprime with the order
of each finite subgroup of $G$. Then $R$ is $G$-simple iff $R\cong\KK^{\alpha}H\otimes M_{\bar{s}}(\KK)$
where $H$ is a finite subgroup of $G$, $\alpha\in Z^{2}(H,\KK^{\times})$,
$\bar{s}\in G^{n}$ and the grading is defined by 
\[
R_{g}=span\left\{ U_{h}\otimes E_{i,j}\;\mid\; s_{i}^{-1}hs_{j}=g\right\} 
\]

\end{thm}

\subsection{Graded Identities}

For a group $G$ we define $X_{g}=\left\{ x_{g,i}\mid i\in\NN\right\} ,\quad X_{G}={\displaystyle \bigcup_{g\in G}X_{g}}$
and $\KK\left\langle X_{G}\right\rangle $ to be the free $\KK$ algebra
generated by the noncommutative indeterminates $X_{G}$. For a monomial
$f={\displaystyle \prod_{j=1}^{k}}x_{g_{j},i_{j}}$we set the degree
$\deg(f)=\prod_{1}^{k}g_{j}$. We define a $G$ grading on $\KK\left\langle X_{G}\right\rangle $
by 
\[
\KK\left\langle X_{G}\right\rangle _{g}=span\left\{ f\in\FF\left\langle X_{G}\right\rangle \;\mid\; f\;\mbox{is a monomial},\;\deg(f)=g\right\} 
\]
A \textit{graded assignment} in $A$ for a polynomial $f\in\KK\left\langle X_{G}\right\rangle $,
is such that each indeterminate $x_{g,i}$ is substituted by an element
from $A_{g}$. The polynomial $f$ is called a \emph{graded identity}
of $A$ if the evaluation of $f$ on any graded assignment is zero,
or in other words $f\mid_{A}=\left\{ 0\right\} $. We define the \textit{ideal
of identities} of $A$ by
\[
Id_{G}(A)=\left\{ f\in\KK\left\langle X_{G}\right\rangle \;\mid\; f\mid_{A}=\left\{ 0\right\} \right\} 
\]
It is easy to check that this is a graded ideal of $\KK\left\langle X_{G}\right\rangle $.
This ideal is also closed under graded endomorphism, so if $f\in Id_{G}(A)$,
then we can substitute an indeterminate $x_{g,i}$ in $f$ by any
polynomial $h\in\KK\left\langle X_{G}\right\rangle $ of degree $g$
and the the result will still be in $Id_{G}(A)$.

A graded ideal in $\KK\left\langle X_{G}\right\rangle $ that is closed
under graded endomorphism is called a \textit{$T$ ideal}.\\

A graded polynomial $f(x_{g_{1},1},...,x_{g_{n},n})$ is called \textit{multilinear}
if it is linear in each of the indeterminates. Every multilinear polynomial
is of the form
\[
\sum_{\sigma\in S_{n}}\lambda_{\sigma}\prod x_{g_{\sigma(i)},\sigma(i)}
\]
If $S$ is some spanning set of homogeneous elements in $A$, then
it is easy to check that a multilinear polynomial $f$ is an identity
of $A$ iff the result of any graded assignment of $S$ in $f$ is
zero. It is well known that for $\KK$ with $char(\KK)=0$, the ideal
$Id_{G}(A)$ is generated as a $T$ ideal by multilinear polynomials,
so in most of the paper we will assume that $char(\KK)=0$ and only
concentrate on multilinear polynomials.

\newpage{}

\section{Graded Embeddings of a Simple Algebra in a Simple Algebra}

In the following section $G$ will be a group, and $A,B$ will be
two finite dimensional $G$-simple algebra over an algebraically closed
field $\KK$ with $char(\KK)=0$.

The question we ask is if the condition $Id_{G}(B)\subseteq Id_{G}(A)$
is enough in order to find a graded embedding $\varphi:A\hookrightarrow B$.
\\
\\
Before we start we give some standard operations on $G$-simple algebras
that return $G$ isomorphic algebras (Lemma 1.3 \cite{aljadeff_simple_2011})
\begin{lem}
Let $H\leq G$ be a finite sub group, $\alpha\in Z^{2}(H,\KK^{\times})$
and $\bar{s}=(s_{1},...,s_{r})\in G^{r}$.
\begin{enumerate}
\item If $\sigma\in S_{r}$ is any permutation and $\bar{s}^{\sigma}\in G^{r}$
is defined by $\bar{s}_{i}^{\sigma}=s_{\sigma(i)}$ then $\KK^{\alpha}H\otimes M_{\bar{s}}(\KK)\cong_{G}\KK^{\alpha}H\otimes M_{\bar{s}^{\sigma}}(\KK)$.
\item Suppose that $Hs_{i}=Ht$. Denote by $\bar{s}'=(s_{1},...,s_{i-1},t,s_{i+1},...,s_{m})$
then $\KK^{\alpha}H\otimes M_{\bar{s}}(K)\cong_{G}\KK^{\alpha}H\otimes M_{\bar{s}'}(K)$.
\item For any $g\in G$ define $g\bar{s}=(gs_{1},...,gs_{n})$, $H_{g}=gHg^{-1}$
and $\alpha_{g}\in Z^{2}(H_{g},\KK^{\times})$ by $\alpha_{g}(gh_{1}g^{-1},gh_{2}g^{-1})=\alpha(h_{1},h_{2})$
then $\KK^{\alpha}H\otimes M_{\bar{s}}(\KK)\cong_{G}\KK^{\alpha_{g}}H_{g}\otimes M_{g\bar{s}}(\KK)$
\end{enumerate}
\end{lem}
\begin{proof}
.
\begin{enumerate}
\item Use the graded isomorphism $U_{h}\otimes E_{i,j}\mapsto U_{h}\otimes E_{\sigma^{-1}(i),\sigma^{-1}(j)}$.
\item Denote by $\tilde{h}\in H$ the element such that $s_{i}=\tilde{h}t$
and then the isomorphism $\varphi:K^{\beta}H\otimes M_{\bar{s}}(K)\rightarrow K^{\beta}H\otimes M_{\bar{s}'}(K)$
is defined by 
\[
\varphi\left(U_{h}E_{j,k}\right)=\begin{cases}
U_{h}E_{j,k} & j,k\neq i\\
\beta(\tilde{h},\tilde{h}^{-1})^{-1}U_{\tilde{h}^{-1}}U_{h}E_{i,k} & j=i,\; k\neq i\\
U_{h}U_{\tilde{h}}E_{j,i} & j\neq i,\; k=i\\
\beta(\tilde{h},\tilde{h}^{-1})^{-1}U_{\tilde{h}^{-1}}U_{h}U_{\tilde{h}}E_{i,i}\qquad & j,k=i
\end{cases}
\]

\item Use the graded isomorphism $\varphi:U_{h}\otimes E_{i,j}\mapsto U_{ghg^{-1}}\otimes E_{i,j}$.
\end{enumerate}
\end{proof}
Parts $1$ and $2$ of the lemma above shows that we can change the
tuple $\bar{s}$ by permuting its elements and by changing elements
in the same right $H$ coset, and it doesn't change the algebra.
\begin{defn}
Let $H\leq G$, and $\bar{s},\bar{s}'\in G^{r}$, then we say that
$\bar{s}$ and $\bar{s}'$ are equivalent modulo $H$ (or just equivalent)
and write $\bar{s}\sim_{H}\bar{s}'$ if we can get $\bar{s}'$ from
$\bar{s}$ by a sequence of permutations, and multiplication from
the left of individual components by elements of $H$. If $\bar{s}$
has a sub tuple that is equivalent to $\bar{s}'$ then we write $\bar{s}\succsim_{H}\bar{s}'$.

Write $\bar{s}\sim\bar{s}'$ (respectively $\bar{s}\succsim\bar{s}'$)
for $\bar{s}\sim_{\left\{ e\right\} }\bar{s}'$ (respectively $\bar{s}\succsim_{\left\{ e\right\} }\bar{s}'$).
\end{defn}
The lemma above shows that if $\bar{s}\sim_{H}\bar{s}'$ and $\alpha$
is any $2$-cocycle of $H$ then $\KK^{\alpha}H\otimes M_{\bar{s}}(\KK)\cong_{G}\KK^{\alpha}H\otimes M_{\bar{s}'}(\KK)$.
If we only have $\bar{s}\succsim_{H}\bar{s}'$ then we get only a
graded embedding $\KK^{\alpha}H\otimes M_{\bar{s}'}(\KK)\hookrightarrow\KK^{\alpha}H\otimes M_{\bar{s}}(\KK)$.

Let $\bar{s}\in G^{r_{1}}$ and $\bar{t}\in G^{r_{2}}$ be two tuples
and define the graded algebra $\KK^{\alpha}H\otimes M_{\bar{s}}(\KK)\otimes M_{\bar{t}}(\KK)$
to be the algebra $\KK^{\alpha}H\otimes M_{r_{1}}(\KK)\otimes M_{r_{2}}(\KK)$
with grading
\[
\deg(U_{h}\otimes E_{i_{1},j_{1}}\otimes E_{i_{2},j_{2}})=t_{i_{2}}^{-1}s_{i_{1}}^{-1}hs_{j_{1}}t_{j_{2}}
\]
Notice that if $G$ is not abelian then we usually don't have $\KK^{\alpha}H\otimes M_{\bar{s}}(\KK)\otimes M_{\bar{t}}(\KK)\cong_{G}\KK^{\alpha}H\otimes M_{\bar{t}}(\KK)\otimes M_{\bar{s}}(\KK)$
with the particular important exception that $\bar{s}=(e,e,..,e)$.
For these two tuples define $\bar{s}\times\bar{t}\in G^{r_{1}r_{2}}$
to be the tuple such that $(\bar{s}\times\bar{t})_{(i,j)}=s_{i}t_{j}$
(up to a permutation) then
\[
\KK^{\alpha}H\otimes M_{\bar{s}\times\bar{t}}(\KK)\cong_{G}\KK^{\alpha}H\otimes M_{\bar{s}}(\KK)\otimes M_{\bar{t}}(\KK)
\]
It is easy to see that if $\bar{s}\sim_{H}\bar{s}'$ then $\bar{s}\times\bar{t}\sim_{H}\bar{s}'\times\bar{t}$.
If $H$ is normal in $G$ then for each $h\in H,\; s,t\in G$ we have
$shs^{-1}\in H$ and $(shs^{-1})st=s(ht)$ , and therefore if $\bar{t}\sim_{H}\bar{t}'$
then $\bar{s}\times\bar{t}\sim_{H}\bar{s}\times\bar{t}'$.

For $\bar{u}\in G^{n},\;\bar{v}\in G^{m}$ define $(\bar{u},\bar{v})=\bar{u}+\bar{v}=(u_{1},...,u_{n},v_{1},...,v_{m})$
then 
\[
(\bar{u}+\bar{v})\times\bar{t}=\bar{u}\times\bar{t}+\bar{v}\times\bar{t}\qquad;\qquad\bar{t}\times(\bar{u}+\bar{v})=\bar{t}\times\bar{u}+\bar{t}\times\bar{v}
\]

For $d\in\NN$ we define $\bar{d}$ to be $\bar{d}:=(\overbrace{e,...,e}^{d\; times})$.

\subsection*{$G$-Envelope}

\global\long\def\dotimes{\widehat{\otimes}}

We now describe a second operation on graded algebras that will be
useful in the proof.

Let $E=E_{0}\oplus E_{1}$ be the infinite Grassmann algebra, and
$A=A_{0}\oplus A_{1}$ be a $\ZZ_{2}$-graded algebra. Denote by $E\dotimes A$
the Grassmann envelope of $A$, with the $\ZZ_{2}$ grading defined
by $\left(E\dotimes A\right)_{0}=E_{0}\otimes A_{0}$ and $\left(E\dotimes A\right)_{1}=E_{1}\otimes A_{1}$.
This is a very important operation in general, and in particular in
PI-theory. One of the main reasons for its importance is that if $B$
is another $\ZZ_{2}$ graded algebra then $Id_{\ZZ_{2}}(A)=Id_{\ZZ_{2}}(B)$
iff $Id_{\ZZ_{2}}(E\dotimes A)=Id_{\ZZ_{2}}(E\dotimes B)$. We now
extend this envelope operation to general groups.
\begin{defn}
[G-envelope]Let $A,B$ be two $G$-graded algebras. We denote by
$A\dotimes B$ to be the $G$-graded algebra defined by $\left(A\dotimes B\right)_{g}=A_{g}\otimes B_{g}$.
\end{defn}
Let $A=\KK^{\alpha}G$ be some twisted group algebra and $B$ a $G$-graded
algebra and denote $B^{\alpha}=A\dotimes B$ and call this algebra
the $\alpha$ envelope of $B$. We claim that this operation enjoys
the same properties as the Grassmann envelope. 

Let $\bar{g}=(g_{1},...,g_{n})\in G^{n}$, and recall that $\alpha(\bar{g})$
is the scalar such that $\prod U_{g_{i}}=\alpha(\bar{g})U_{\prod g_{i}}$.
For $\sigma\in S_{n}$ denote $\bar{g}^{\sigma}=(g_{\sigma(1)},...,g_{\sigma(n)})$.
For a tuple $\bar{g}\in G^{n}$, $g\in G$, define $S_{n}^{\bar{g},g}=\left\{ \sigma\in S_{n}\;\mid\;\prod g_{\sigma(i)}=g\right\} $,
then any multilinear polynomial in the indeterminates $\left(x_{g_{i},i}\right)_{1}^{n}$
which is homogeneous of degree $g$ is of the form $f(x_{g_{1},1},...,x_{g_{n},n})={\displaystyle \sum_{\sigma\in S_{n}^{\bar{g},g}}}\lambda_{\sigma}\prod x_{g_{\sigma(i)},\sigma(i)}$.
Notice that if $G$ is abelian then $S_{n}^{\bar{g},g}=S_{n}$ whenever
$g=\prod g_{i}$, and is empty otherwise. The polynomial $f$ is an
identity of $B^{\alpha}$ iff for all $b_{i}\in B_{g_{i}}$ we have
\begin{eqnarray*}
0 & = & \sum_{\sigma}\lambda_{\sigma}\prod_{i}\left(U_{g_{\sigma(i)}}\otimes b_{i}\right)=U_{g}\otimes\sum_{\sigma}\lambda_{\sigma}\alpha(\bar{g}^{\sigma})\prod_{i}b_{\sigma(i)}\\
 & \iff & 0=\sum_{\sigma}\lambda_{\sigma}\alpha(\bar{g}^{\sigma})\prod_{i}b_{\sigma(i)}
\end{eqnarray*}
Set $f^{\alpha}=\sum_{\sigma}\lambda_{\sigma}\alpha(\bar{g}^{\sigma})\prod_{i}x_{\sigma(i)}$
then we just showed that $f\in Id_{G}(B^{\alpha})$ iff $f^{\alpha}\in Id_{G}(B)$. 
\begin{lem}
\label{lem:cocycle envelope}Let $B_{1},B_{2}$ be two $G$-graded
algebra and $\alpha\in Z^{2}(G,\KK^{\times})$ then
\begin{enumerate}
\item $\left(B^{\alpha}\right)^{\alpha^{-1}}\cong_{G}B$.
\item There is a graded embedding $B_{1}\hookrightarrow_{G}B_{2}$ iff there
is a graded embedding $B_{1}^{\alpha}\hookrightarrow_{G}B_{2}^{\alpha}$.
\item There is an inclusion $Id_{G}(B_{1})\subseteq Id_{G}(B_{2})$ iff
$Id_{G}(B_{1}^{\alpha})\subseteq Id_{G}(B_{2}^{\alpha})$.
\end{enumerate}
\end{lem}
\begin{proof}
.
\begin{enumerate}
\item Write $\KK^{\alpha}G=\bigoplus\KK U_{g}$ and $\KK^{\alpha^{-1}}G=\bigoplus\KK V_{g}$,
then $\psi:\left(B^{\alpha}\right)^{\alpha^{-1}}\rightarrow B$ defined
by $\psi(V_{g}\otimes U_{g}\otimes b)=b$ for $b\in B_{g}$ is a graded
isomorphism.
\item If $\varphi:B_{1}\rightarrow B_{2}$ then $\varphi^{\alpha}:B_{1}^{\alpha}\rightarrow B_{2}^{\alpha}$
defined by $\varphi^{\alpha}(U_{g}\otimes b)=U_{g}\otimes\varphi(b)$
for $b\in B_{g}$ is a graded embedding. The other direction is proved
using part (1).
\item Assume that $f\in Id_{G}(B_{1}^{\alpha})$ then $f^{\alpha}\in Id_{G}(B_{1})\subseteq Id_{G}(B_{2})$
so $f\in Id_{G}(B_{2}^{\alpha})$.
\end{enumerate}
\end{proof}
A standard proof shows that if $B$ is $G$-simple then $B^{\alpha}$
is also $G$-simple. In the case where $B$ is finite dimensional
we know exactly what is $B^{\alpha}$.
\begin{thm}
\label{thm:Envelope simple}Let $B=\KK^{\beta}H\otimes M_{\bar{s}}(\KK)$
then $B^{\alpha}\cong_{G}\KK^{\beta\cdot\alpha}H\otimes M_{\bar{s}}(\KK)$.\end{thm}
\begin{proof}
Let $\KK^{\alpha}G=\bigoplus\KK U_{g}$ , $\KK^{\beta}H=\bigoplus\KK V_{h}$
and $\KK^{\beta\cdot\alpha}H=\bigoplus\KK W_{h}$. We define $\psi:B^{\alpha}\rightarrow\KK^{\beta\cdot\alpha}G\otimes M_{\bar{s}}(\KK)$
by 
\[
\psi(U_{s_{i}^{-1}gs_{j}}\otimes V_{g}\otimes E_{i,j})=W_{g}\otimes\frac{\sqrt{\alpha(s_{i}^{-1},s_{i})\alpha(s_{j}^{-1},s_{j})}}{\alpha(s_{i}^{-1},g,s_{j})}E_{i,j}
\]
We leave the interested reader to show that this map multiplicative.
\end{proof}
Let $E$ be a trivially graded algebra then the algebra $\KK^{\alpha}G\otimes E$
has a natural $G$-grading. Similar to the $G$-envelope, one can
show a connection between $E$ and graded algebra $\KK^{\alpha}G\otimes E$.
\begin{lem}
\label{lem:cocycle envelope-ungraded}Let $E$ be trivially graded
algebras, and $\alpha\in Z^{2}(G,\KK^{\times})$. Let $\bar{g}=\left(g_{1},...,g_{n}\right)\in G^{n}$
, $g\in G$ then $f(x_{g_{1}},...,x_{g_{n}})={\displaystyle \sum_{\sigma\in S_{n}^{\bar{g},g}}}\lambda_{\sigma}\prod x_{g_{\sigma(i)},\sigma(i)}$
is a graded identity of $\KK^{\alpha}G\otimes E$ iff 
\[
f^{\bar{g},\alpha}(x_{1},...,x_{n}):=\sum_{\sigma\in S_{n}^{\bar{g},g}}\lambda_{\sigma}\alpha(\bar{g}^{\sigma})\prod x_{\sigma(i)}\in Id(E)
\]
\end{lem}
\begin{proof}
Similar to the proof with $\alpha$ envelope.
\end{proof}

\subsection{Part 1 - $A=\KK^{\alpha}H\otimes M_{\bar{s}}(\KK)$ , $B=\KK^{\beta}G\otimes M_{r_{2}}(\KK)$
, $H\leq G$ and $\bar{s}\in G^{r_{1}}$}

We now start to build the graded embeddings. In all the following
steps, we assume that $G$ is abelian and $\KK$ is an algebraically
closed field of characteristic zero. We will always have a twisted
group algebra $\KK^{\alpha}G_{1}$ in $A$ and $\KK^{\beta}G_{2}$
in $B$, and we denote the basis of $\KK^{\alpha}G_{1}$ with $V_{g}$
and the basis for $\KK^{\beta}G_{2}$ with $U_{g}$ for $g\in G$.
\begin{thm}
\label{thm:tuple only left}Let $G$ be a finite abelian group, $H\leq G$
a subgroup, $\bar{s}\in G^{r_{1}}$ a tuple and $\alpha\in Z^{2}(H,\KK^{\times}),\;\beta\in Z^{2}(G,\KK^{\times})$.
Let $\beta'=\beta\mid_{H}$, and let $d$ be the dimension of the
smallest representation of $\KK^{\alpha/\beta'}H$ then for $A=\KK^{\alpha}H\otimes M_{\bar{s}}(\KK)$
and $B=\KK^{\beta}G\otimes M_{r_{2}}(\KK)$ the following are equal
\begin{enumerate}
\item There is a graded embedding $A\hookrightarrow B$
\item $Id_{G}(B)\subseteq Id_{G}(A)$
\item $r_{2}\geq d\cdot\left|\bar{s}\right|=d\cdot r_{1}$, or equivalently
$\overline{r_{2}}\succsim_{G}\bar{d}\times\bar{s}$
\end{enumerate}
\end{thm}
Taking the $\beta^{-1}$ envelope on both algebras we can use \prettyref{lem:cocycle envelope}
to see that part $(1)$ is equivalent to a graded embedding $A^{\beta^{-1}}\hookrightarrow B^{\beta^{-1}}$
and part $(2)$ is equivalent to $Id_{G}(B^{\beta^{-1}})\subseteq Id_{G}(A^{\beta^{-1}})$.
By \prettyref{thm:Envelope simple} we get that $B^{\beta^{-1}}\cong_{G}\KK G\otimes M_{r_{2}}(\KK)$
and $A^{\beta^{-1}}\cong_{G}\KK^{\alpha/\beta}H\otimes M_{\bar{s}}(\KK)$,
so part $(3)$ remains without change. This gives us a reduction to
the problem where the cocycle in $B$ is trivial, and from now on
we assume that $\beta=1$.

Assume first that $A=\KK^{\alpha}G\otimes M_{r_{1}}(\KK)$, so the
matrix part of $A$ is trivially graded. $\beta$ is trivial so we
need a way to map the possibly non-trivial cocycle $\alpha$ to $B$,
but since $\beta$ is trivial then we will have to compensate using
the matrix part of $B$. In order to prove the theorem we first give
a lemma that shows the compensation costs a $d\times d$ matrix algebra
in $B$ where $d$ is the dimension of the smallest representation
of $\KK^{\alpha}H$.
\begin{lem}
\label{lem:smallest representation}Let $G$ be an arbitrary finite
group, $r_{1},r_{2}\in\NN$, $\alpha\in Z^{2}(G,\KK^{\times})$, and
denote by $d$ the dimension of the smallest representation of $\KK^{\alpha}G$,
then 
\begin{enumerate}
\item There is a graded embedding $\KK^{\alpha}G\otimes M_{r_{1}}(\KK)\hookrightarrow\KK G\otimes M_{r_{2}}(\KK)$
iff $r_{2}\geq r_{1}d$. 
\item If $G$ is abelian and $Id_{G}\left(\KK G\otimes M_{r_{2}}(\KK)\right)\subseteq Id_{G}\left(\KK^{\alpha}G\otimes M_{r_{1}}(\KK)\right)$
then $r_{2}\geq r_{1}d$.
\end{enumerate}
\end{lem}
\begin{proof}

\begin{enumerate}
\item Denote by $\epsilon:\KK G\rightarrow\KK$ the augmentation representation
and by $\rho:\KK^{\alpha}G\rightarrow M_{d}(\KK)$ the smallest representation
of $\KK^{\alpha}G$. \\
Suppose first that there is a graded embedding $\varphi:\KK^{\alpha}G\otimes M_{r_{1}}(\KK)\hookrightarrow\KK G\otimes M_{r_{2}}(\KK)$
and compose it with $\epsilon\otimes id$, then we have a map $\KK^{\alpha}G\otimes M_{r_{1}}(\KK)\rightarrow M_{r_{2}}(\KK)$.
The function $\varphi$ is a graded embedding so $\varphi(V_{e}\otimes I)=U_{e}\otimes a$
for some $0\neq a\in M_{r_{2}}(\KK)$ and therefore $\epsilon\circ\varphi(V_{e}\otimes I)=\epsilon(U_{e})\otimes a=a\neq0$
so $\epsilon\circ\varphi\neq0$. Decompose $\KK^{\alpha}G$ to a direct
sum of matrix algebras of dimensions $d_{1}^{2}\leq d_{2}^{2}\leq\cdots\leq d_{k}^{2}$
where $d_{1}=d$. Since $\epsilon\circ\varphi\neq0$, at least one
of these matrix algebras, which are simple, is mapped injectively
into $M_{r_{2}}(\KK)$ so there is $1\leq i\leq k$ such that $d^{2}r_{1}^{2}\leq d_{i}^{2}r_{1}^{2}\leq r_{2}^{2}$
and therefore $dr_{1}\leq r_{2}$.\\
Suppose now that $dr_{1}\leq r_{2}$, then we have an embedding $\psi:M_{d}(\KK)\otimes M_{r_{1}}(\KK)\rightarrow M_{r_{2}}(\KK)$.
Define $\varphi:\KK^{\alpha}G\otimes M_{r_{1}}(\KK)\rightarrow\KK G\otimes M_{r_{2}}(\KK)$
by 
\[
\varphi(\sum_{g\in G}V_{g}\otimes a_{g})=\sum_{g\in G}U_{g}\otimes\psi(\rho(V_{g})\otimes a_{g})
\]
then $\varphi$ is linear, but also multiplicative since
\begin{eqnarray*}
\varphi\left(\left(V_{g}\otimes a_{g}\right)\left(V_{h}\otimes b_{h}\right)\right) & = & \varphi\left(\alpha(g,h)V_{gh}\otimes a_{g}b_{h}\right)=\alpha(g,h)\left(U_{gh}\otimes\psi(\rho(V_{gh})\otimes a_{g}b_{h})\right)\\
 & = & U_{g}U_{h}\otimes\psi(\rho(V_{g}V_{h})\otimes a_{g}b_{h})=\left(U_{g}\otimes\psi(\rho(V_{g})\otimes a_{g})\right)\left(U_{h}\otimes\psi(\rho(V_{h})\otimes b_{h})\right)\\
 & = & \varphi\left(V_{g}\otimes a_{g}\right)\varphi\left(V_{h}\otimes b_{h}\right)
\end{eqnarray*}
$\varphi$ is a graded non-zero algebra homomorphism and $A$ is $G$-simple
so $\varphi$ is a graded embedding.
\item Let $\alpha\in Z^{2}(G,\KK^{\times})$ and let $H=\left\{ h\in G\;\mid\;\frac{\alpha(h,g)}{\alpha(g,h)}=1\;\forall g\in G\right\} $
be the degrees of homogeneous elements in the center of $\KK^{\alpha}G$,
then as a non-graded algebra we have (\cite{aljadeff_graded_2011}
Corollary 13)
\[
\KK^{\alpha}G\cong\overbrace{M_{d}(\KK)\oplus M_{d}(\KK)\oplus\cdots\oplus M_{d}(\KK)}^{\left|H\right|\; times}\qquad d^{2}=\left[G:H\right]
\]
$\KK G$ is semisimple abelian algebra, so as a non-graded algebra
it is a direct sum of the field $\KK$ and we get
\begin{eqnarray*}
\KK^{\alpha}G\otimes M_{r_{1}}(\KK) & \cong & \overbrace{M_{dr_{1}}(\KK)\oplus M_{dr_{1}}(\KK)\oplus\cdots\oplus M_{dr_{1}}(\KK)}^{\left|H\right|\; times}\\
\KK G\otimes M_{r_{2}}(\KK) & \cong & \overbrace{M_{r_{2}}(\KK)\oplus M_{r_{2}}(\KK)\oplus\cdots\oplus M_{r_{2}}(\KK)}^{\left|G\right|\; times}\\
(*)\quad Id(M_{r_{2}}(\KK))=Id(\KK G\otimes M_{r_{2}}(\KK)) & \subseteq & Id(\KK^{\alpha}G\otimes M_{r_{1}}(\KK))=Id(M_{dr_{1}}(\KK))
\end{eqnarray*}
From Amitsur-Levitzki theorem\cite{amitsur_minimal_1950} the multilinear
identity of $M_{n}(\KK)$ with the smallest degree, has degree $2n$,
and from the inclusion above we see that $2dr_{1}\leq2r_{2}\;\Rightarrow\; dr_{1}\leq r_{2}$.
\end{enumerate}
\end{proof}
Since we can assume that $\beta\equiv1$, then the last lemma proves
that $(1)\iff(3)$ and $(2)\Rightarrow(3)$ in \ref{thm:tuple only left}.
$(1)\Rightarrow(2)$ is obvious, and so the theorem is proved for
the case where $A=\KK^{\alpha}G\otimes M_{r_{1}}(\KK)$.\\

Before we continue, we note that \prettyref{lem:smallest representation}
part (1) shows that if $d$ is the dimension of the smallest representation
of $\KK^{\alpha}G$, then there is a graded embedding $\varphi:\KK^{\alpha}G\rightarrow\KK G\otimes M_{d}(\KK)$.
As non-graded algebras $\KK^{\alpha}G$ and $\KK G\otimes M_{d}(\KK)$
are direct sum of matrix algebras 
\begin{eqnarray*}
\KK^{\alpha}G & \cong & M_{d_{1}}(\KK)\oplus\cdots\oplus M_{d_{n}}(\KK)\\
\KK G\otimes M_{d}(\KK) & \cong & \left(M_{r_{1}}(\KK)\oplus\cdots\oplus M_{r_{m}}(\KK)\right)\otimes M_{d}(\KK)\cong M_{dr_{1}}(\KK)\oplus\cdots\oplus M_{dr_{m}}\\
 &  & d=d_{1}\leq d_{2}\leq\cdots\leq d_{n}\quad;\quad1=r_{1}\leq r_{2}\leq\cdots\leq r_{m}
\end{eqnarray*}
Let $\pi_{j}:\KK G\otimes M_{d}(\KK)\rightarrow M_{dr_{j}}(\KK)$
be the natural projection. $M_{d_{n}}(\KK)$ is not in the kernel
of $\varphi$ (because it is injective), so there is some $j$ such
that $M_{d_{n}}(\KK)$ is not in the kernel of $\pi_{j}\circ\varphi$,
and since it is simple, then $\pi_{j}\circ\varphi\mid_{M_{d_{n}}(\KK)}$
is injective. From this we get that $d_{n}\leq dr_{j}\leq dr_{m}\quad\Rightarrow\quad d_{n}/d\leq r_{m}$.
By definition $d=d_{1}$, and $r_{1}=1$ because we have the augmentation
representation on $\KK G$, so we just proved that:
\begin{thm}
Let $G$ be a finite arbitrary group. For $\alpha\in Z^{2}(G,\KK^{\times})$
we define $\Phi(\alpha)$ to be the ratio $\frac{d_{n}}{d_{1}}$ where
$d_{n}$ is the dimension of the largest representation of $\KK^{\alpha}G$
and $d_{1}$ is the dimension of the smallest representation, then
$\Phi$ achieves its maximum at $\alpha\equiv1$. 
\end{thm}
We return now to the general case. Suppose that $r_{2}\geq d\left|\bar{s}\right|$
as in part $(3)$ in the theorem then $\overline{r_{2}}\succsim_{G}\bar{d}\times\bar{s}$
since all the elements of $\bar{s}$ are in $G$. This means that
up to this equivalence we can {}``find'' the $M_{\bar{s}}(\KK)$
part of $A$ in the matrix part of $B$.

If this idea is true, then the choice of the tuple $\bar{s}$ is irrelevant,
and only its size matters. Changing $\bar{s}$ to the tuple $\left(e,e,...,e\right)$
to get the algebra $A'=\KK^{\alpha}H\otimes M_{r_{1}}(\KK)$ we reduce
the grading of $A$ to the subgroup $H$. Doing the same with $B$,
we need to reduce the group algebra part from $\KK G$ to $\KK H$,
and this will give a reduction to the first case.
\begin{lem}
Let $G$ be a finite abelian group, $H\leq G$. Let $A=\KK^{\alpha}H\otimes M_{\bar{s}}(\KK),\; B=\KK G\otimes M_{r_{2}}(\KK)$
where $\bar{s}\in G^{r_{1}}$ and define $A'=\KK^{\alpha}H\otimes M_{r_{1}}(\KK)$,
$B'=\KK H\otimes M_{r_{2}}(\KK)=B_{H}$. 

If $Id_{G}(B)\subseteq Id_{G}(A)$ then $Id_{H}(B')\subseteq Id_{H}(A')$.\end{lem}
\begin{proof}
We first recall from \prettyref{lem:cocycle envelope-ungraded} that
any multilinear polynomial identity of $B'$ has the form $f(x_{h_{1},1},...,x_{h_{n},n})={\displaystyle \sum_{\sigma\in S_{n}}}\lambda_{\sigma}\prod x_{h_{\sigma(i)},\sigma(i)}$
where $\bar{h}=(h_{1},...,h_{n})\in H^{n}$ and $\tilde{f}=\sum\lambda_{\sigma}\prod x_{\sigma(i)}\in Id(M_{r_{2}}(\KK))$
(we use here that $\beta\equiv1$). We want to show that $f$ is a
graded identity of $A'$. From multilinearity, $f\in Id_{H}(\KK^{\alpha}H\otimes M_{r_{1}}(\KK))$
iff for every $a_{i},b_{i}\in\left[r_{1}\right],\;1\leq i\leq n$
we have 
\[
(*)\qquad f\left(V_{h_{1}}\otimes E_{a_{1},b_{1}},...,V_{h_{n}}\otimes E_{a_{n},b_{n}}\right)=0
\]
Notice that the algebras $A$ and $A'$ are the same and only differ
in their grading, so we can think of the assignment above in $A$
instead of $A'$. Fix now the indices $a_{i},b_{i}$, then the grading
of the assignment above in $A$ is $g_{i}=\deg_{A}\left(V_{h_{i}}\otimes E_{a_{i},b_{i}}\right)=s_{a_{i}}^{-1}h_{i}s_{b_{i}}$.
For the same $f$ as above let $\bar{g}=(g_{1},...,g_{n})\in G^{n}$
then using again \prettyref{lem:cocycle envelope-ungraded} and the
fact that $\beta\equiv1$ we get that $f(x_{g_{1},1},...,x_{g_{n},n})={\displaystyle \sum_{\sigma\in S_{n}}}\lambda_{\sigma}\prod x_{g_{\sigma(i)},\sigma(i)}$
is an identity of $B$ , because $\tilde{f}\in Id(M_{r_{2}}(\KK))$.
This shows that $(*)$ is indeed zero in $A$, so also in $A'$. We
can do this for any assignment in $A'$ of basis elements (though
for each assignment, the tuple $\bar{g}$ may be different). $f$
is multilinear, so $f\in Id_{G}(A')$ and we are done.
\end{proof}
We are now ready to prove \prettyref{thm:tuple only left}
\begin{proof}
The direction $(1)\Rightarrow(2)$ is obvious.

Suppose we have $(2)$ so $Id_{G}(\KK G\otimes M_{r_{2}}(\KK))\subseteq Id_{G}(\KK^{\alpha}H\otimes M_{\bar{s}}(\KK))$.
From the last lemma we have $Id_{G}(\KK H\otimes M_{r_{2}}(\KK))\subseteq Id_{G}(\KK^{\alpha}H\otimes M_{\left|\bar{s}\right|}(\KK))$
and we can now use \prettyref{lem:smallest representation} (2) to
conclude that $r_{2}\geq d\left|s\right|$. Notice that since $\bar{s}\in G^{r_{1}}$
then $\bar{d}\times\bar{s}\sim_{G}\bar{d}\times\overline{r_{1}}$,
so $\overline{r_{2}}\succsim_{G}\bar{d}\times\bar{s}$ and we proved
$(2)\Rightarrow(3)$.

Assume $(3)$. From \prettyref{lem:smallest representation} (1) we
have a graded embedding $\iota:\KK^{\alpha}H\hookrightarrow\KK H\otimes M_{d}(\KK)$.
Since $\overline{r_{2}}\succsim_{G}\bar{d}\times\bar{s}$, we can
assume wlog that $B=\KK G\otimes M_{d}(\KK)\otimes M_{\bar{s}}(\KK)$.
Define the function $\varphi:A\rightarrow B$ by
\[
\varphi(V_{h}\otimes E_{i,j})=\iota(V_{h})\otimes E_{i,j}
\]
This is a graded algebra homomorphism since
\begin{eqnarray*}
\deg(\iota(V_{h})\otimes E_{i,j}) & = & s_{i}^{-1}\deg(\iota(V_{h}))s_{j}=s_{i}^{-1}hs_{j}=\deg(V_{h}\otimes E_{i,j})\\
\varphi(V_{h_{1}}\otimes E_{i_{1},j_{1}})\varphi(V_{h_{2}}\otimes E_{i_{2},j_{2}}) & = & \left(\iota(V_{h_{1}})\otimes E_{i_{1},j_{1}}\right)\left(\iota(V_{h_{2}})\otimes E_{i_{2},j_{2}}\right)=\iota(V_{h_{1}}V_{h_{2}})\otimes E_{i_{1},j_{2}}\delta_{j_{1},i_{2}}\\
 & = & \alpha(h_{1},h_{2})\iota(V_{h_{1}h_{2}})\otimes E_{i_{1},j_{2}}\delta_{j_{1},i_{2}}=\alpha(h_{1},h_{2})\delta_{i_{2},j_{1}}\varphi(V_{h_{1}h_{2}}\otimes E_{i_{1},j_{2}})\\
 & = & \varphi(V_{h_{1}}V_{h_{2}}\otimes E_{i_{1},j_{1}}E_{i_{2},j_{2}})=\varphi(\left(V_{h_{1}}\otimes E_{i_{1},j_{1}}\right)\left(V_{h_{2}}\otimes E_{i_{2},j_{2}}\right))
\end{eqnarray*}
$\varphi\neq0$ and $A$ is simple graded so $\varphi$ is a graded
embedding hence $(3)\Rightarrow(1)$
\end{proof}

\subsection{Part 2 - $A=\KK^{\alpha}N_{1}\otimes M_{\bar{s}}(\KK)$ , $B=\KK^{\beta}N_{2}\otimes M_{\bar{t}}(\KK)$
, $\bar{s}\in N_{2}^{r_{1}},\;\bar{t}\in N_{1}^{r_{2}}$}

Before we go on to the proof of this part we give a notation for better
use of the structure of $A$ and $B$ as block matrices. Let $N_{1}\leq G'\leq G$
and $\bar{s}\in G^{r_{1}}$ then we can decompose $\bar{s}$ into
$\bar{s}=\bar{s}^{(1)}+\bar{s}^{(2)}+\cdots+\bar{s}^{(n)}$ such that
$\bar{s}^{(i)}$ is a tuple of elements in $G'g_{i}$, and $\left\{ g_{1},...,g_{n}\right\} $
are different right $G'$ coset representatives (in this part we will
take $G'=N_{1}$). With this fixed decomposition, for the algebra
$A=\KK^{\alpha}N_{1}\otimes M_{\bar{s}}(\KK)$ we write
\begin{eqnarray*}
M_{g_{i},g_{j}} & = & M_{\bar{s}^{(i)},\bar{s}^{(j)}}:=span\left\{ E_{k,l}\;\mid\; s_{k}\in G'g_{i},\; s_{l}\in G'g_{j}\right\} \\
A_{g_{i},g_{j}} & = & \KK^{\alpha}N_{1}\otimes M_{g_{i},g_{j}}
\end{eqnarray*}
an easy check shows that $M_{g_{i},g_{j}}M_{g_{k},g_{l}}=0$ if $j\neq k$
and $M_{g_{i},g_{j}}M_{g_{j},g_{l}}\subseteq M_{g_{i},g_{l}}$. We
also have $\deg(A_{g_{i},g_{j}})\subseteq g_{i}^{-1}G'g_{j}=g_{i}^{-1}g_{j}G'$
which is a coset of $G'$.

This gives the decompositions $M_{r_{1}}(\KK)=\bigoplus M_{g_{i},g_{j}}$
and $A=\bigoplus A_{g_{i},g_{j}}$ to block matrices. We assume that
$G$ is abelian (though it is enough to assume that $G'\trianglelefteq G$)
then the blocks on the diagonal have all degrees in $\deg(A_{g_{i},g_{i}})=g_{i}^{-1}G'g_{i}=G'$.
Suppose that we can find $h\in g_{i}^{-1}G'g_{j}\cap G'$ so there
is some $\tilde{h}\in G'$ such that 
\[
h=g_{i}^{-1}\tilde{h}g_{j}\quad\Rightarrow\quad hg_{i}=\tilde{h}g_{j}
\]
so $g_{i}$ and $g_{j}$ are in the same right coset of $G'$, hence
they are equal. This shows that $A_{G'}$ is exactly the diagonal
$\bigoplus_{1}^{n}A_{g_{i},g_{i}}$.

We now check how the rest of the cosets of $G'$ sit inside of $A$.
Notice that 
\[
\deg(A_{g_{i},g_{j}})\cap\deg(A_{g_{i},g_{k}})\subseteq g_{i}^{-1}G'g_{j}\cap g_{i}^{-1}G'g_{k}=g_{i}^{-1}\left(G'g_{j}\cap G'g_{k}\right)
\]
which is empty unless $g_{j}=g_{k}$, so in particular a coset of
$G'$ can be in at most one block on the $i$'th row, and similarly
for columns.\\

Suppose now that $N_{1}N_{2}=G$ where $N_{1},N_{2}\leq G$, $A=\KK^{\alpha}N_{1}\otimes M_{\bar{s}}(\KK)$
, $B=\KK^{\beta}N_{2}\otimes M_{\bar{t}}(\KK)$ , $\bar{s}\in N_{2}^{r_{1}},\;\bar{t}\in N_{1}^{r_{2}}$.
We want to see how $A_{N_{2}}$ (respectively $B_{N_{1}}$) sits inside
$A$ (respectively $B$). Suppose that $\deg_{A}\left(V_{h}\otimes E_{i,j}\right)=s_{i}^{-1}hs_{j}\in N_{2}$
for some $h\in N_{1},\; s_{i},s_{j}\in N_{2}$, then $h\in s_{i}N_{2}s_{j}^{-1}=N_{2}$,
and therefore $h\in N_{1}\cap N_{2}$. On the other hand, if $h\in N_{1}\cap N_{2}$
and $s_{i},s_{j}\in N_{2}$ then $\deg_{A}\left(V_{h}\otimes E_{i,j}\right)\in N_{2}$,
and we see that $A_{N_{2}}=\KK^{\alpha\mid_{H}}H\otimes M_{\bar{s}}(\KK)$
where $H=N_{1}\cap N_{2}$. A similar argument shows that $B_{N_{1}}=\KK^{\beta\mid_{H}}H\otimes M_{\bar{t}}(\KK)$.

To help visualize the structure of $A$ (and $B$) think of it as
a three dimensional cube, where the face in the front is $M_{\bar{s}}(\KK)$
and as we move on the axis perpendicular to it we change the coefficient
in $\KK^{\alpha}N_{1}$, then $A$ (and $B$) will look like

\begin{center}
\includegraphics[scale=0.55]{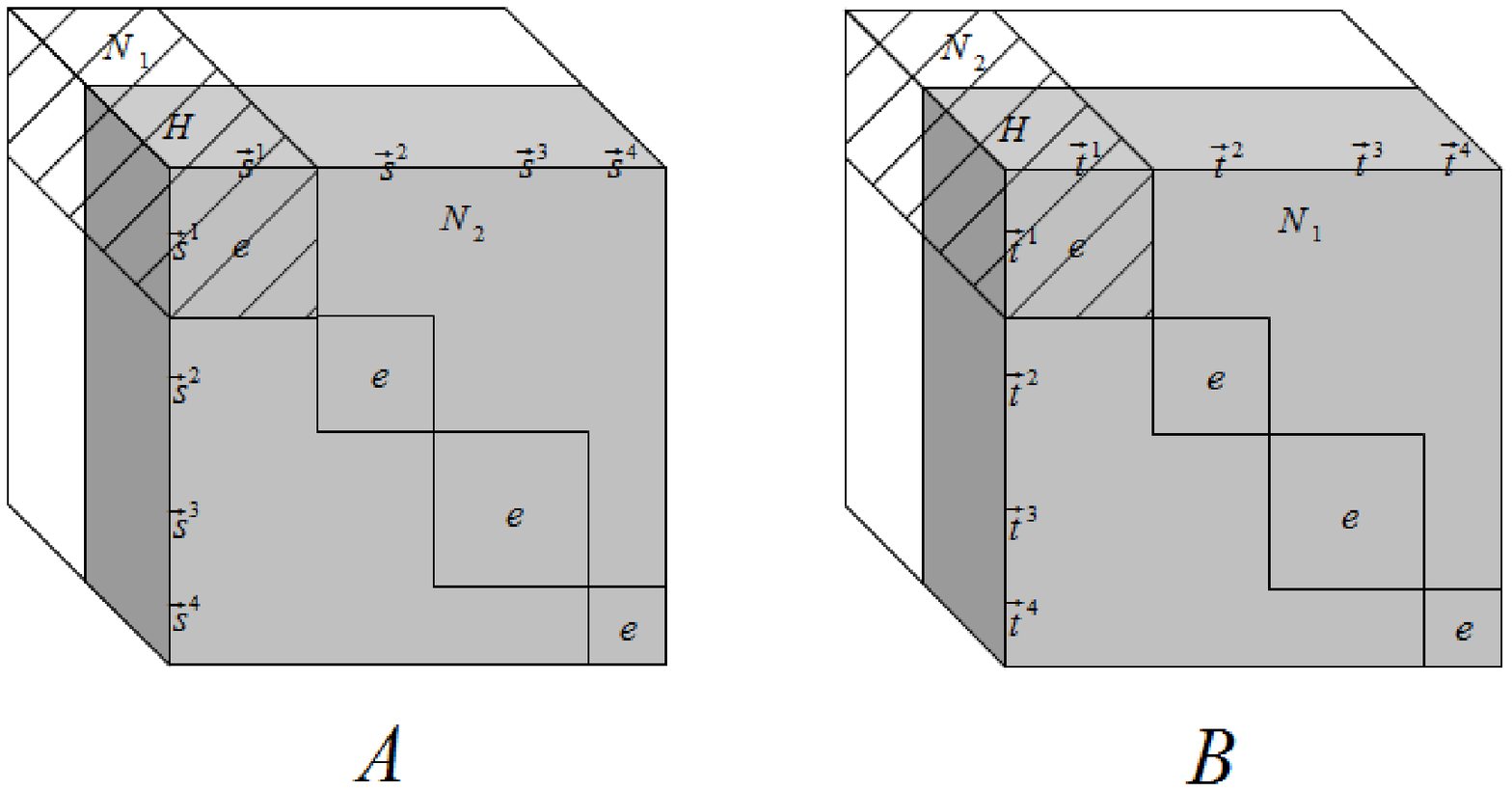}
\par\end{center}

In the last part, we had only the gray part in $A$, and only the
striped part in $B$. To extend our embedding we need to find where
to send the rest of $N_{1}$. If $\gamma$ is a $2$-cocycle on $G$
and the tuple $\bar{g}=\left(g_{1},...,g_{n}\right)$ has all the
elements in $G$ exactly once, then there is a natural way to embed
$\KK^{\gamma}G$ inside $M_{\bar{g}}(\KK)$ using the right regular
representation on $\KK^{\gamma}G$. The extension here is done in
a similar manner, where we think of $\nicefrac{N_{1}}{H}$ as inside
$\KK^{\alpha}N_{1}$ acting by right multiplication on $A$ (with
the exact definitions inside the theorem). In this case $\nicefrac{N_{1}}{H}$
will sit inside $M_{\bar{g}}(\KK)$ where $\bar{g}\in N_{1}^{\left[N_{1}:H\right]}$
is a tuple of coset representatives of $N_{1}$ in $H$ which is also
a tuple of coset representatives of $N_{2}$ in $G=N_{1}N_{2}$. We
denote by $\overline{G:N_{2}}$ a tuple of coset representatives of
$N_{2}$ in $G$. Notice that up to equivalence modulo $N_{2}$, all
such tuples are the same.
\begin{thm}
Let $G$ be a finite abelian group, $N_{1},N_{2}\leq G$ such that
$N_{1}N_{2}=G$ , $\alpha\in Z^{2}(N_{1},\KK^{\times})$, $\beta\in Z^{2}(N_{2},\KK^{\times})$
and $\bar{s}\in N_{2}^{r_{1}},\;\bar{t}\in N_{1}^{r_{2}}$. Denote
$H=N_{1}\cap N_{2}$, $\alpha'=\alpha\mid_{H}$ , $\beta'=\beta\mid_{H}$
and let $d$ be the dimension of the smallest representation of $\KK^{\alpha'/\beta'}H$.
Define $A=\KK^{\alpha}N_{1}\otimes M_{\bar{s}}(\KK)$ and $B=\KK^{\beta}N_{2}\otimes M_{\bar{t}}(\KK)$
then the following are equal
\begin{enumerate}
\item There is a graded embedding $A\hookrightarrow B$
\item $Id_{G}(B)\subseteq Id_{G}(A)$
\item $\bar{t}\succsim_{N_{2}}\bar{d}\times\bar{s}\times\overline{G:N_{2}}$.
\end{enumerate}
\end{thm}
\begin{proof}
The part $(1)\Rightarrow(2)$ is obvious.

Assume now that $Id_{G}(B)\subseteq Id_{G}(A)$. Let $T=\left\{ w_{1},...,w_{n}\right\} \subseteq N_{1}$
a set of coset representatives of $N_{2}$ in $G=N_{1}N_{2}$ (containing
the elements of $\bar{t}$) and denote $B_{w_{i}}=\KK^{\beta}N_{2}\otimes M_{w_{i},w_{i}}\subseteq B_{N_{2}}$
and $A'=\KK^{\alpha'}H\otimes M_{\bar{s}}(\KK)=A_{N_{2}}$. In the
figure above, $A'$ is the gray part, while the $B_{w_{i}}$ are the
block on the diagonal. 

The idea that we use here and in the next part is that if $B$ has
$n\times n$ blocks and $a_{1},...,a_{n}$ are diagonal block matrices
such that the $i$'th block of $a_{i}$ is zero then $\prod a_{i}=0$.
To use this we will try to find polynomials $f_{i}$, such that for
each assignment in $B$, the $i$'th block will be zero, and since
$B_{N_{2}}=\bigoplus_{1}^{n}B_{w_{i}}$ this is equivalent to $f_{i}\in Id_{N_{2}}(B_{w_{i}})$. 

Suppose now that $a$ is a block diagonal matrix with $0$ on the
$i$ block. Let $b,c$ be two block matrices such that the only non-zero
blocks on the $\tilde{i}$'th column (as a block matrix) of $b$ is
in location $(i,\tilde{i})$ and the only non-zero block on the $\tilde{i}$'th
row of $c$ is on the $(\tilde{i},i)$ location then 
\[
(cab)_{\tilde{i},\tilde{i}}=\sum_{j,k}c_{\tilde{i},j}a_{j,k}b_{k,\tilde{i}}=c_{\tilde{i},i}a_{i,i}b_{i,\tilde{i}}=0
\]
Let $g\in G$ be fixed and $b\in B_{gN_{2}}$. The degrees on the
$i$'th row are $w_{i}^{-1}N_{2}w_{j}$ where $1\leq j\leq n$ so
$b$ can have a non-zero block only in location $(i,j)$ where $w_{i}^{-1}w_{j}N_{2}=gN_{2}\;\iff\; w_{j}\in w_{i}gN_{2}$.
By the remark before the theorem, $b$'s only non-zero block on the
$j$ column can be only in the $(i,j)$ location. A similar argument
is true for $c\in B_{g^{-1}N_{2}}$, and if $a\in A_{e}$ is a diagonal
block matrix with zero $i$ block then $cab\in A_{e}$ is again a
block diagonal matrix with a zero $j$ block, so we can move the zero
block from the $i$'th location to the $j$'th location. This is true
for any such $b\in B_{gN_{2}}$ and $c\in B_{g^{-1}N_{2}}$ so we
can extend this argument to the level of identities.

Back to the proof, we know that $Id_{G}(B)\subseteq Id_{G}(A)$, so
in particular we have
\[
\bigcap Id_{N_{2}}(B_{w_{i}})=Id_{N_{2}}\left(\bigoplus B_{w_{i}}\right)=Id_{N_{2}}(B_{N_{2}})\subseteq Id_{N_{2}}(A_{N_{2}})=Id_{N_{2}}(A')
\]
We want to show that not only the intersection of $Id_{N_{2}}(B_{w_{i}})$
is in $Id_{N_{2}}(A')$, but for each $i$ there is an inclusion $Id_{N_{2}}(B_{w_{i}})\subseteq Id_{N_{2}}(A')$.
Suppose by negation that there is some $1\leq i_{0}\leq n$, $f\in Id_{N_{2}}(B_{w_{i_{0}}})$
such that $f\notin Id_{N_{2}}(A')$. Since $B_{N_{2}}=\bigoplus_{1}^{n}B_{w_{i}}$
are the block diagonal matrices of $B$, then for any assignment of
$B$ in $f$ the $w_{i_{0}}$ block will be zero. Let $f_{j}=x_{w_{j}^{-1}w_{i_{0}}}\cdot f\cdot y_{w_{i_{0}}^{-1}w_{j}}$
then for any assignment of $B$ in $f_{j}$ the $w_{j}$ block on
the diagonal will be zero, and so $\tilde{f}=\prod_{1}^{n}\left(f_{j}\cdot z_{N_{2},j}\right)$
will be an identity of $B$ because each element of $B_{N_{2}}$ is
a diagonal block matrix, and each block on the diagonal will be zero. 

On the other hand $w_{j}^{-1}w_{i_{0}}\in N_{1}$ for each $j$ and
$V_{w_{j}^{-1}w_{i_{0}}}\otimes I\in A_{w_{j}^{-1}w_{i_{0}}}$ are
invertibles so $f_{j}\notin Id_{G}(A)$ for every $j$. For each $1\leq a,b\leq\left|s\right|$
and $h\in H$ we have $\deg(V_{h}\otimes E_{a,b})\in N_{2}$ so we
can find an assignment for $\prod_{1}^{n}f_{j}\cdot z_{N_{2},j}$
in $A$ that is not zero, hence $\tilde{f}\notin Id_{G}(A)$, and
we get a contradiction because $Id_{G}(B)\subseteq Id_{G}(A)$.

From the contradiction we know that $Id_{N_{2}}(B_{w_{i}})\subseteq Id_{N_{2}}(A')$
for each $i$. Write $\bar{t}=\bar{t}^{(1)}+\cdots+\bar{t}^{(n)}$
where the elements of $\bar{t}^{(i)}$ are in $N_{2}w_{i}$ . Notice
that in order to use the previous part we needed that the elements
of $\bar{t}^{(i)}$ are all $e$, or up to equivalence modulo $N_{2}$
all the elements are in $N_{2}$. Taking $\bar{t}^{(i)}w_{i}^{-1}$
will satisfy this condition, and since $G$ is abelian we have $B_{w_{i}}=\KK^{\beta}N_{2}\otimes M_{\bar{t}^{(i)}}(\KK)\cong_{G}\KK^{\beta}N_{2}\otimes M_{\bar{t}^{(i)}w_{i}^{-1}}(\KK)$.
We now use the previous part to show that $\left|\bar{t}^{(i)}\right|=\left|\bar{t}^{(i)}w_{i}^{-1}\right|\geq\left|\bar{s}\right|\cdot d$
for each $i$ therefore 
\begin{eqnarray*}
\bar{t}^{(i)}w_{i}^{-1}\succsim_{N_{2}}\bar{s}\times\bar{d}\quad & \Rightarrow & \quad\bar{t}^{(i)}\succsim_{N_{2}}\bar{s}\times\bar{d}\times\left(w_{i}\right)\\
\bar{t}=\bar{t}^{(1)}+\cdots+\bar{t}^{(n)}\succsim_{N_{2}}\sum_{i}\bar{s}\times\bar{d}\times\left(w_{i}\right) & = & \bar{s}\times\bar{d}\times\sum_{i}\left(w_{i}\right)=\bar{d}\times\bar{s}\times\overline{G:N_{2}}
\end{eqnarray*}
 so $(2)\Rightarrow(3)$.\\

Suppose that condition $(3)$ is true, then wlog we can assume that
the tuple $\bar{t}$ is $\bar{d}\times\bar{s}\times\overline{G:N_{2}}$
or in other words $M_{\bar{t}}(\KK)\cong M_{d}(\KK)\otimes M_{\bar{s}}(\KK)\otimes M_{\overline{G:N_{2}}}(\KK)$.

As we mentioned in the remark before the theorem, we now think of
$\KK^{\alpha}N_{1}\otimes M_{\bar{s}}(\KK)$ as a $\left[G:N_{2}\right]$
copies of $\KK^{\alpha'}H\otimes M_{\bar{s}}(\KK)$, and then the
right multiplication in $\KK^{\alpha}N_{1}\otimes M_{\bar{s}}(\KK)$
permutes these copies (that are mapped to the diagonal block matrices
in $B$) and also act on each copy alone. 
\[
N_{1}=N_{1}\cap N_{1}N_{2}=N_{1}\cap\left(\uplus N_{2}w_{i}\right)=\uplus\left(N_{1}\cap N_{2}w_{i}\right)=\uplus\left(N_{1}\cap N_{2}\right)w_{i}
\]
so $T$ is also a set of representatives of $H=N_{1}\cap N_{2}$ in
$N_{1}$. 

We first give a notation. Let $w_{i}\in T$ then for every $g\in N_{1}$
there are unique $h\in H$ and $w_{j}\in T$ such that $w_{i}g=hw_{j}$.
Denote these $h,w_{j}$ by $h:=h_{w_{i},g}$ and $w_{j}=w_{i}^{g}$
(in particular $w_{i}^{e}=w_{i}$ for all $i$) then 
\begin{eqnarray*}
w_{i}(g_{1}g_{2}) & = & h_{w_{i},g_{1}g_{2}}w_{i}^{g_{1}g_{2}}\\
\left(w_{i}g_{1}\right)g_{2} & = & h_{w_{i},g_{1}}w_{i}^{g_{1}}g_{2}=h_{w_{i},g_{1}}h_{w_{i}^{g_{1}},g_{2}}(w_{i}^{g_{1}})^{g_{2}}\\
 & \Rightarrow & h_{w_{i},g_{1}}h_{w_{i}^{g_{1}},g_{2}}=h_{i,g_{1}g_{2}}\qquad w_{i}^{g_{1}g_{2}}=(w_{i}^{g_{1}})^{g_{2}}
\end{eqnarray*}
and in particular $G\rightarrow Aut(T)$ defined above is a homomorphism.
Notice that for $N_{1}$ abelian and $g\in H$ we get that $w_{i}g=gw_{i}$
so $w_{i}^{g}=w_{i}$ and $h_{w_{i},g}=g$ for any $w_{i}\in T$.

We would like to view $N_{1}$ as a group that acts on itself where
{}``itself'' is $\sqcup Hw_{i}$. Each element in $N_{1}$ can be
written uniquely as $hw_{i}$ where $w_{i}\in T$ and $h\in H$. To
make calculations easier, we choose a new basis for $\KK^{\alpha}N_{1}$,
such that $V_{w_{i}}$ can be any choice, except $V_{e}=1$, and $V_{h}$
can be any choice for $h\in H\backslash\left\{ e\right\} $ and then
for each $h\in H$ and $w_{i}\in T$ we let $V_{h}V_{w_{i}}=V_{hw_{i}}$,
or in other words $\alpha(h,w_{i})=1$ for every $h\in H$ and $w_{i}\in T$. 

For $w_{i}\in T,\; h\in H$ , $g\in N_{1}$ and $b,c\in M_{\bar{s}}(\KK)$
we have
\begin{eqnarray*}
\left(V_{h}V_{w_{i}}\otimes b\right)\left(V_{g}\otimes c\right) & = & \alpha(w_{i},g)V_{h}V_{w_{i}g}\otimes bc=\alpha(w_{i},g)V_{h}V_{h_{w_{i},g}w_{i}^{g}}\otimes bc=\alpha(w_{i},g)\left(V_{h}V_{h_{w_{i},g}}\otimes bc\right)\left(V_{w_{i}^{g}}\otimes1\right)
\end{eqnarray*}
From the previous part, we have a graded embedding $\varphi:\KK^{\alpha'}H\otimes M_{\bar{s}}(\KK)\rightarrow\KK^{\beta}N_{2}\otimes M_{\left|s\right|d}(\KK)$.
We now define a mapping 
\begin{eqnarray*}
\phi:\KK^{\alpha}N_{1}\otimes M_{\bar{s}}(\KK) & \rightarrow & \KK^{\beta}N_{2}\otimes M_{\bar{t}}(\KK)\cong\KK^{\beta}N_{2}\otimes M_{d\left|s\right|}(\KK)\otimes M_{\overline{G:N_{2}}}(\KK)\\
\phi(V_{g}\otimes c) & = & {\displaystyle \sum_{w_{i}\in T}}\alpha(w_{i},g)\varphi\left(V_{h_{w_{i},g}}\otimes c\right)\otimes E_{w_{i},w_{i}^{g}}
\end{eqnarray*}
where we identify in the $M_{\overline{G:N_{2}}}(\KK)$ part the set
of indices $\left\{ 1,...,\left[G:N_{2}\right]\right\} $ with the
set $T$ of coset representatives. $\varphi$ is graded so the degree
of each summand is 
\[
w_{i}^{-1}h_{w_{i},g}\deg_{A}(c)w_{i}^{g}=w_{i}^{-1}w_{i}g\cdot\deg_{A}(c)=g\cdot\deg_{A}(c)=\deg_{A}(V_{g}\otimes c)
\]
hence $\phi$ is graded. The linear extension of this map is a vector
space homomorphism and we want it to be also multiplicative. We write
$V_{g}$ for $V_{g}\otimes I$ and then
\begin{eqnarray*}
\phi(V_{g_{1}})\phi(V_{g_{2}}) & = & \left({\displaystyle \sum_{w_{i}\in T}}\alpha(w_{i},g_{1})\varphi\left(V_{h_{w_{i},g_{1}}}\right)\otimes E_{w_{i},w_{i}^{g_{1}}}\right)\left({\displaystyle \sum_{w_{j}\in T}}\alpha(w_{j},g_{2})\varphi\left(V_{h_{w_{j},g_{2}}}\right)\otimes E_{w_{j},w_{j}^{g_{2}}}\right)\\
 & = & {\displaystyle \sum_{w_{i}\in T}}\alpha(w_{i},g_{1})\alpha(w_{i}^{g_{1}},g_{2})\varphi\left(V_{h_{w_{i},g_{1}}}\right)\varphi\left(V_{h_{w_{i}^{g_{1}},g_{2}}}\right)\otimes E_{w_{i},\left(w_{i}^{g_{1}}\right)^{g_{2}}}\\
 & = & {\displaystyle \sum_{w_{i}\in T}}\alpha(w_{i},g_{1})\alpha(w_{i}^{g_{1}},g_{2})\alpha(h_{w_{i},g_{1}},h_{w_{i}^{g_{1}},g_{2}})\varphi\left(V_{h_{w_{i},g_{1}g_{2}}}\right)\otimes E_{w_{i},w_{i}^{g_{1}g_{2}}}=(*)
\end{eqnarray*}
We now use the 2-cocycle property of $\alpha$ together with $\alpha(h,w_{i})=1$
for all $h\in H$ and $w_{i}\in T$ to get 
\begin{eqnarray*}
\alpha(hw_{j},g) & = & \alpha(hw_{j},g)\alpha(h,w_{j})=\alpha(h,w_{j}g)\alpha(w_{j},g)=\alpha(h,h_{w_{j},g}w_{j}^{g})\alpha(w_{j},g)\\
\alpha(h,h_{w_{j},g}w_{j}^{g}) & = & \alpha(h,h_{w_{j},g}w_{j}^{g})\alpha(h_{w_{j},g},w_{j}^{g})=\alpha(h,h_{w_{j},g})\alpha(hh_{w_{j},g},w_{j}^{g})=\alpha(h,h_{w_{j},g})\\
\Rightarrow\alpha(h,h_{w_{j},g})\alpha(w_{j},g) & = & \alpha(hw_{j},g)\\
\Rightarrow\alpha(h_{w_{i},g_{1}},h_{w_{i}^{g_{1}},g_{2}})\alpha(w_{i}^{g_{1}},g_{2}) & = & \alpha(h_{w_{i},g_{1}}w_{i}^{g_{1}},g_{2})=\alpha(w_{i}g_{1},g_{2})
\end{eqnarray*}
\begin{eqnarray*}
(*) & = & {\displaystyle {\displaystyle \sum_{w_{i}\in T}}}\alpha(w_{i},g_{1})\alpha(w_{i}g_{1},g_{2})\varphi\left(V_{h_{w_{i},g_{1}g_{2}}}\right)\otimes E_{w_{i},w_{i}^{g_{1}g_{2}}}=\alpha(g_{1},g_{2}){\displaystyle \sum_{w_{i}\in T}}\alpha(w_{i},g_{1}g_{2})\varphi\left(V_{h_{w_{i},g_{1}g_{2}}}\right)\otimes E_{w_{i},w_{i}^{g_{1}g_{2}}}\\
 & = & \alpha(g_{1},g_{2})\phi(V_{g_{1}g_{2}})=\phi(V_{g_{1}}V_{g_{2}})
\end{eqnarray*}
From the definition of $\phi$ we get that $\phi(V_{g}\otimes c)=\phi(V_{g})\left(\varphi(1\otimes c)\otimes I\right)=\left(\varphi(1\otimes c)\otimes I\right)\phi(V_{g})$
(using the multiplicativity of $\varphi$) and then
\begin{eqnarray*}
\phi(V_{g_{1}}\otimes c_{1})\phi(V_{g_{2}}\otimes c_{2}) & = & \phi(V_{g_{1}})\varphi(1\otimes c_{1})\phi(V_{g_{2}})\varphi(1\otimes c_{2})=\phi(V_{g_{1}})\phi(V_{g_{2}})\varphi(1\otimes c_{1})\varphi(1\otimes c_{2})\\
 & = & \phi(V_{g_{1}}V_{g_{2}})\varphi(1\otimes c_{1}c_{2})=\phi(V_{g_{1}}V_{g_{2}}\otimes c_{1}c_{2})
\end{eqnarray*}
so $\phi$ is multiplicative.

$\phi$ is a non-zero graded algebra homomorphism from $A$ to $B$
and $A$ is $G$-simple so $\phi$ is a graded embedding, and this
proves part $(3)\Rightarrow(1)$ and completes the theorem. Notice
that if we view $B$ as 
\[
B\cong\KK^{\beta}N_{2}\otimes M_{d}(\KK)\otimes M_{\left|\bar{s}\right|}(\KK)\otimes M_{\overline{G:N_{2}}}(\KK)\cong\KK^{\beta}N_{2}\otimes M_{d}(\KK)\otimes M_{\bar{s}}(\KK)\otimes M_{\overline{G:N_{2}}}(\KK)
\]
 and $\iota:\KK^{\alpha'}H\rightarrow\KK^{\beta'}H\otimes M_{d}(\KK)$
the graded embedding from step 1 then
\begin{eqnarray*}
\phi(V_{g}\otimes c) & = & {\displaystyle \sum_{w_{i}\in T}}\alpha(w_{i},g)\varphi\left(V_{h_{w_{i},g}}\otimes c\right)\otimes E_{w_{i},w_{i}^{g}}={\displaystyle \sum_{w_{i}\in T}}\alpha(w_{i},g)\iota(V_{h_{w_{i},g}})\otimes c\otimes E_{w_{i},w_{i}^{g}}\\
\phi(V_{e}\otimes c) & = & {\displaystyle \sum_{w_{i}\in T}}\alpha(w_{i},e)\varphi\left(V_{h_{w_{i},e}}\otimes c\right)\otimes E_{w_{i},w_{i}^{e}}={\displaystyle \sum_{w_{i}\in T}}\varphi\left(V_{e}\otimes c\right)\otimes E_{w_{i},w_{i}}=\iota(V_{e})\otimes c\otimes I
\end{eqnarray*}
 
\end{proof}

\subsection{Part 3 - $A=\KK^{\alpha}N_{1}\otimes M_{\bar{s}}(\KK)$ , $B=\KK^{\beta}N_{2}\otimes M_{\bar{t}}(\KK)$ }

As we did in the previous part, we are going to think of $A$ and
$B$ as block matrices, such that $G'=N_{1}N_{2}$ will be the blocks
on the diagonal. We will then try to match each such block of degree
$G'$ in $A$ to some block with degree $G'$ in $B$ and then {}``extend''
the embedding to all of $A$.
\begin{thm}
\label{thm:Main Theorem}Let $G$ be an abelian group, $N_{1},N_{2}\leq G$
be finite subgroups of $G$ , $\alpha\in Z^{2}(N_{1},\KK^{\times})$,
$\beta\in Z^{2}(N_{2},\KK^{\times})$ and $\bar{s}\in G^{r_{1}},\;\bar{t}\in G^{r_{2}}$.
Let $H=N_{1}\cap N_{2}$, $\alpha'=\alpha\mid_{H}$ , $\beta'=\beta\mid_{H}$
and let $d$ be the dimension of the smallest representation of $\KK^{\alpha'/\beta'}H$,
then the following conditions are equal for $A=\KK^{\alpha}N_{1}\otimes M_{\bar{s}}(\KK)$
and $B=\KK^{\beta}N_{2}\otimes M_{\bar{t}}(\KK)$
\begin{enumerate}
\item There is a graded embedding $A\hookrightarrow B$
\item $Id_{G}(B)\subseteq Id_{G}(A)$
\item $\exists g\in G,\quad$$g\bar{t}\succsim_{N_{2}}d\times\overline{G':N_{2}}\times\tilde{s}$
where $\bar{s}\sim_{N_{1}}\tilde{s}$.
\end{enumerate}
\end{thm}
\begin{proof}
The part $(1)\Rightarrow(2)$ is obvious.

Suppose that $(3)$ is true so we can assume wlog that $\bar{t}=d\times\overline{G':N_{2}}\times\bar{s}$
and therefore $B=\KK^{\beta}N_{2}\otimes M_{d}(\KK)\otimes M_{\overline{G':N_{2}}}(\KK)\otimes M_{\bar{s}}(\KK)$.
From the previous part, we can find a graded embedding $\phi_{1}:\KK^{\alpha}N_{1}\rightarrow\KK^{\beta}N_{2}\otimes M_{d}(\KK)\otimes M_{\overline{N_{1}N_{2}:N_{2}}}(\KK)$
. Define the map $\phi:A\rightarrow B$ by
\[
\phi(V_{g}\otimes E_{i,j})=\phi_{1}(V_{g})\otimes E_{i,j}
\]
and extend linearly. This is a graded homomorphism since
\begin{eqnarray*}
\deg_{B}\left(\phi_{1}(V_{g})\otimes E_{i,j}\right) & = & s_{i}^{-1}\deg\left(\phi_{1}(V_{g})\right)s_{j}=s_{i}^{-1}gs_{j}=\deg_{A}(V_{g}\otimes E_{i,j})
\end{eqnarray*}
where we used the fact that $\phi_{1}$ is graded.
\begin{eqnarray*}
\phi(V_{g_{1}}\otimes E_{i_{1},j_{1}})\phi(V_{g_{2}}\otimes E_{i_{2},j_{2}}) & = & \left(\phi_{1}(V_{g_{1}})\otimes E_{i_{1},j_{1}}\right)\cdot\left(\phi_{1}(V_{g_{2}})\otimes E_{i_{2},j_{2}}\right)=\delta_{j_{1},i_{2}}\left(\phi_{1}(V_{g_{1}})\phi_{1}(V_{g_{2}})\otimes E_{i_{1},j_{2}}\right)\\
 & = & \delta_{j_{1},i_{2}}\alpha(g_{1},g_{2})\left(\phi_{1}(V_{g_{1}g_{2}})\otimes E_{i_{1},j_{2}}\right)=\delta_{j_{1},i_{2}}\alpha(g_{1},g_{2})\phi\left(V_{g_{1}g_{2}}\otimes E_{i_{1},j_{2}}\right)\\
 & = & \phi\left(\left(V_{g_{1}}\otimes E_{i_{1},j_{1}}\right)\left(V_{g_{2}}\otimes E_{i_{2},j_{2}}\right)\right)
\end{eqnarray*}
$\phi$ is a non-zero graded algebra homomorphism and $A$ is $G$-simple
so it is a graded embedding.\\

Suppose now that condition $(2)$ is true - $Id_{G}(B)\subseteq Id_{G}(A)$.
Decompose $\bar{s}$ to $\bar{s}=\left(\bar{s}^{(1)},...,\bar{s}^{(n)}\right)$
such that the elements of $\bar{s}^{(i)}$ (non-empty tuple) are in
$G'u_{i}$ and $U=\left\{ u_{1},...,u_{n}\right\} $ are different
coset representatives of $G'$ in $G$, and define $\bar{t}=\left(\bar{t}^{(1)},...,\bar{t}^{(m)}\right)$
and $V=\left\{ v_{1},...,v_{m}\right\} $ similarly. We write 
\begin{eqnarray*}
A_{i} & := & \KK^{\alpha}N_{1}\otimes M_{\bar{s}^{(i)},\bar{s}^{(i)}}\cong_{G}\KK^{\alpha}N_{1}\otimes M_{\bar{s}^{(i)}u_{i}^{-1},\bar{s}^{(i)}u_{i}^{-1}}\\
B_{i} & := & \KK^{\beta}N_{2}\otimes M_{\bar{t}^{(i)},\bar{t}^{(i)}}\cong_{G}\KK^{\beta}N_{2}\otimes M_{\bar{t}^{(i)}v_{i}^{-1},\bar{t}^{(i)}v_{i}^{-1}}
\end{eqnarray*}
and we want to show that $m\geq n$ and up to an injective function
$\sigma:\left[m\right]\rightarrow\left[n\right]$ we have $Id_{G'}(B_{\sigma(i)})\subseteq Id_{G'}(A_{i})$.
Notice that $A_{G'}=\bigoplus A_{i}$ and $B_{G'}=\bigoplus B_{j}$.

For each $1\leq i\leq n$ and $1\leq j\leq m$ if $Id_{G'}(B_{j})\not\subseteq Id_{G'}(A_{i})$
then choose some polynomial $f_{i,j}\in Id_{G'}(B_{j})\backslash Id_{G'}(A_{i})$.
For $i$ fixed define $f_{i}=x_{G',0}\cdot\prod_{j}\left[f_{i,j}\cdot x_{G',j}\right]$,
then for $1\leq j_{0}\leq m$ if $Id_{G'}(B_{j_{0}})\not\subseteq Id_{G'}(A_{i})$
then $f_{i,j_{0}}$ appears in this product, and therefore, for every
assignment of $B$ in $f_{i}$, the $j_{0}$ block on the diagonal
will be zero. In other words, if there is an assignment of $f_{i}$
such that the $j_{0}$ block is not zero then $Id_{G'}(B_{j})\subseteq Id_{G'}(A_{i})$.

Define 
\[
f=\prod_{i=1}^{n}y_{u_{1}^{-1}u_{i}G',i}\cdot f_{i}\cdot z_{u_{i}^{-1}u_{1}G',i}
\]

For each $i,j$, the polynomial $f_{i,j}\notin Id_{G'}(A_{i})$ and
$A_{i}$ is simple so $f_{i}\notin Id_{G'}(A_{i})$ (\prettyref{lem:product in simple algebras})
and therefore there is an assignment $\bar{a}_{i}$ for $f_{i}$ such
that the $u_{i}$ block on the diagonal is not zero. An assignment
in $y_{u_{1}^{-1}u_{i}G',i}\cdot f_{i}\cdot z_{u_{i}^{-1}u_{1}G',i}$
from $A$ will move the $u_{i}$ block on the diagonal to the $u_{1}$
block, and since 
\begin{eqnarray*}
\deg\left(V_{e}\otimes M_{u_{1},u_{i}}\right) & = & u_{1}u_{i}^{-1},\qquad\deg\left(V_{e}\otimes M_{u_{j},u_{1}}\right)=u_{i}^{-1}u_{1}\\
M_{u_{1},u_{i}} & = & M_{u_{1},u_{1}}M_{u_{1},u_{i}}
\end{eqnarray*}
then $M_{u_{1},u_{i}}f(\bar{a}_{i})M_{u_{i},u_{1}}$ has a non-zero
$u_{1}$ block and then
\[
\prod_{1}^{n}\left[M_{u_{1},u_{1}}\left(M_{u_{1},u_{i}}f(\bar{a}_{i})M_{u_{i},u_{1}}\right)\right]
\]
still has a non-zero $u_{1}$ block so $f\notin Id_{G}(A)$ and therefore
$f\notin Id_{G}(B)$.

A general assignment of $B$ in $y_{u_{1}^{-1}u_{i}G',i}\cdot f_{i}\cdot z_{u_{i}^{-1}u_{1}G',i}$
moves the $v_{j}$ block after the assignment in $f_{i}$ to $v_{j_{0}}$
block where $v_{j_{0}}^{-1}v_{j}G'=u_{1}^{-1}u_{i}G'$ if there is
such a $j_{0}$, and other wise this block goes to zero. $f\notin Id_{G}(B)$
so there is a block $v_{j_{0}}$ which is not zero for some assignment
in $f$. Let $\sigma(i)$ be the index such that $v_{\sigma(i)}G'=u_{1}^{-1}u_{i}v_{j_{0}}G'$,
then from the remark above the $\sigma(i)$ block is not zero in the
assignment of $f_{i}$, so $Id_{G'}(B_{\sigma(i)})\subseteq Id_{G'}(A_{i})$.
Notice that 
\[
\sigma(i_{1})=\sigma(i_{2})\;\Rightarrow\; u_{1}^{-1}u_{i_{1}}v_{j_{0}}G'=u_{1}^{-1}u_{i_{2}}v_{j_{0}}G'\;\Rightarrow\; u_{i_{1}}G'=u_{i_{2}}G'\;\Rightarrow\; i_{1}=i_{2}
\]
because the $u_{i}$ are different coset representatives, so $\sigma$
is injective and wlog we can assume that $\sigma(i)=i$. We now move
to a graded sub algebra of $B$ by looking only on the sub tuple $\bar{t}=(\bar{t}^{(1)},...,\bar{t}^{(n)})$.
Let $g\in G$ such that $v_{1}g=u_{1}$ then 
\begin{eqnarray*}
v_{1}G' & = & u_{1}^{-1}u_{1}v_{j_{0}}G'\;\Rightarrow\; v_{1}G'=v_{j_{0}}G'\;\Rightarrow\; v_{j_{0}}=v_{1}\\
v_{i}G' & = & u_{1}^{-1}u_{i}v_{1}G'=u_{i}g^{-1}G'
\end{eqnarray*}
so by multiplying $\bar{t}$ by $g$ we see that the elements of $g\bar{t}^{(i)}$
and $\bar{s}^{(i)}$ are in in the same coset $u_{i}G'$. 

$g\bar{t}^{(i)}u_{i}^{-1}$ and $\bar{s}^{(i)}u_{i}^{-1}$ are in
$G'=N_{1}N_{2}$ and so we can find $g\bar{t}^{(i)}u_{i}^{-1}\sim_{N_{2}}\tilde{t}^{(i)}$
and $\bar{s}^{(i)}u_{i}^{-1}\sim_{N_{1}}\tilde{s}^{(i)}$ such that
the elements of $\tilde{t}^{(i)}$ are in $N_{1}$ and the elements
of $\tilde{s}^{(i)}$ are in $N_{2}$. From the previous part, since
$Id_{G'}(B_{i})\subseteq Id_{G'}(A_{i})$ we get
\[
g\bar{t}^{(i)}u_{i}^{-1}\sim_{N_{2}}\tilde{t}^{(i)}\succsim_{N_{2}}\bar{d}\times\tilde{s}^{(i)}\times\overline{G':N_{2}}
\]
\begin{eqnarray*}
g\bar{t} & = & g\bar{t}^{(1)}+\cdots+g\bar{t}^{(n)}=g\bar{t}^{(1)}u_{1}^{-1}u_{1}+\cdots+g\bar{t}^{(n)}u_{n}^{-1}u_{n}\\
 & \succsim_{N_{2}} & \left(\bar{d}\times\tilde{s}^{(1)}\times\overline{G':N_{2}}\right)u_{1}+\cdots+\left(\bar{d}\times\tilde{s}^{(n)}\times\overline{G':N_{2}}\right)u_{n}=\bar{d}\times(\tilde{s}^{(1)}u_{1}+\cdots\tilde{s}^{(n)}u_{n})\times\overline{G':N_{2}}\\
\bar{s} & = & (\bar{s}^{(1)}+\cdots\bar{s}^{(n)})=(\bar{s}^{(1)}u_{1}^{-1}u_{1}+\cdots\bar{s}^{(n)}u_{n}^{-1}u_{n})\sim_{N_{1}}\left(\tilde{s}^{(1)}u_{1}+\cdots+\tilde{s}^{(n)}u_{n}\right)
\end{eqnarray*}
so $(2)\Rightarrow(3)$, and the theorem is proved.
\end{proof}
By the structure theorem of finite dimensional $G$-simple algebras,
all such algebras are of the form $\KK^{\alpha}H\otimes M_{\bar{s}}(\KK)$
where $H\leq G$ and $\bar{s}\in G^{r}$ for some $r$ so we have:
\begin{thm}
Let $A,B$ be two finite dimensional $G$-simple algebra, where $G$
is an abelian group, then there is a graded embedding $A\hookrightarrow B$
iff $Id_{G}(B)\subseteq Id_{G}(A)$.
\end{thm}
\newpage{}

\section{The Non-Abelian Case}

We concentrate now on a special case of \prettyref{thm:Main Theorem}
- the one where $N_{2}=\left\{ e\right\} $. The theorem shows that
if $A=\KK^{\alpha}H\otimes M_{\bar{s}}(\KK)$ and $B=M_{\bar{t}}(\KK)$
where $G$ is abelian, then there is a graded embedding $A\hookrightarrow B$
iff $Id_{G}(B)\subseteq Id_{G}(A)$ iff there exists $g\in G$ such
that $g\cdot\bar{t}\succsim\bar{H}\times\tilde{s},\;\tilde{s}\sim_{H}\bar{s}$
where $\bar{H}$ is a tuple such that each element in $H$ appears
in $\bar{H}$ exactly once. We now show that in this case we can omit
the condition of commutativity.
\begin{thm}
Let $A=\KK^{\alpha}H\otimes M_{\bar{s}}(\KK)$, $B=M_{\bar{t}}(\KK)$
be $G$-graded algebras, $G$ arbitrary, then the following are equal
\begin{enumerate}
\item There is a graded embedding $A\hookrightarrow B$
\item $Id_{G}(B)\subseteq Id_{G}(A)$
\item $\exists g\in G$, s.t. $g\cdot\bar{t}\succsim\overline{H}\times\tilde{s}$
and $\tilde{s}\sim_{H}\bar{s}$.
\end{enumerate}
\end{thm}
\begin{proof}
$(1)\Rightarrow(2)$ is obvious.

Suppose we have $Id_{G}(B)\subseteq Id_{G}(A)$. Write $\tilde{s}\sim_{H}\bar{s}=\left(s_{1}\times r_{1},....,s_{n}\times r_{n}\right)$
where $\left\{ s_{i}\right\} _{1}^{n}$ are distinct right coset representatives
of $H$. For each $1\leq i\leq n$ and each $h\in H$ we define 
\[
f_{i,h}=x_{s_{1}^{-1}hs_{i}}\cdot St_{2r_{i}-1}(z_{e,1},...,z_{e,2r_{i}-1})\cdot y_{(s_{1}^{-1}hs_{i})^{-1}}
\]
From the Amitsur-Levitzki theorem, there is an assignment for the
$z_{e,j}$ such that the $s_{i}$ block on the diagonal of $A$ is
not zero in $St_{2r_{i}-1}(z_{e,1},...,z_{e,2r_{i}-1})$, and by choosing
suitable assignment for $x_{s_{1}^{-1}hs_{i}}$ and $y_{(s_{1}^{-1}hs_{i})^{-1}}$
we see that the $s_{1}$ block on the diagonal is not zero in $f_{i,h}$.
Define
\[
f=\prod_{i=1}^{n}\;\prod_{h\in H}\; f_{i,h}\cdot x_{e,(i,h)}
\]
where each $f_{i,h}$ is considered with different indeterminates.
Since there is an assignment for each $f_{i,h}$ where the $s_{1}$
block is not zero, and it is isomorphic to $M_{r_{1}}(\KK)$ which
is simple, we get that $f$ is not an identity of $A$.

$Id_{G}(B)\subseteq Id_{G}(A)$ so $f\notin Id_{G}(B)$. $f$ has
degree $e$, so if $\bar{t}\sim(t_{1}\times k_{1},...,t_{m}\times k_{m})$
where the $\left\{ t_{i}\right\} _{1}^{m}$ are distinct, then there
is an assignment $\bar{b}$ for $f$ and some $j$ such that the $t_{j}$
block is not zero in $f(\bar{b})$ and by permuting the elements of
$\bar{t}$ we can assume that $j=1$. The $t_{1}$ block on the diagonal
of $f(\bar{b})$ is not zero, so it cannot be zero in $f_{i,h}(\bar{b})$
for each $i,h$, and this can be true only if the $t_{\sigma(i,h)}$
block on the diagonal of $St_{2r_{i}-1}(\bar{b})$ is not zero where
$t_{1}^{-1}t_{\sigma(i,h)}=s_{1}^{-1}hs_{i}$. By the Amitsur-Levitzki
theorem we get that $k_{\sigma(i,h)}\geq r_{i}$. Notice that $\sigma(i,h)$
is injective since
\[
\sigma(i,h)=\sigma(i',h')\;\Rightarrow\; s_{1}^{-1}hs_{i}=s_{1}^{-1}h's_{i'}\;\Rightarrow\; hs_{i}=h's_{i'}
\]
but the $s_{i}$ are distinct $H$ right coset representatives so
$s_{i}=s_{i'}\iff i=i'$ and $h=h'$. We therefore have
\[
\bar{t}\succsim\sum_{i}\sum_{h\in H}\left(t_{\sigma(i,h)}\times k_{\sigma(i,h)}\right)\succsim\sum_{i}\sum_{h\in H}\left(t_{1}s_{1}^{-1}hs_{i}\times r_{i}\right)=t_{1}s_{1}^{-1}\sum_{i}\sum_{h\in H}\left(hs_{i}\times r_{i}\right)=t_{1}s_{1}^{-1}\cdot\left(\bar{H}\times\bar{s}\right)
\]
and the required $g$ for $(3)$ is $s_{1}t_{1}^{-1}$. 

Suppose that $\left(3\right)$ is true, so wlog we can assume that
$\bar{t}=\bar{H}\times\bar{s}$. Let $\varphi:\KK^{\alpha}H\rightarrow M_{\overline{H}}(\KK)$
be the right regular representation defined by
\[
U_{h}\mapsto\sum_{h_{i}\in H}\alpha(h_{i},h)E_{h_{i},h_{i}h}
\]
 where we identify $\left\{ 1,...,\left|H\right|\right\} $ with $H$.
Define the algebra homomorphism $\psi:A\rightarrow B$ by $\psi(U_{h}\otimes E_{k,l})=\varphi(U_{h})\otimes E_{k,l}$.
\[
\deg\left(\psi(U_{h}\otimes E_{k,l})\right)=\deg\left(\varphi(U_{h})\otimes E_{k,l}\right)=s_{k}^{-1}\deg\left(\varphi(U_{h})\right)s_{l}=s_{k}^{-1}hs_{l}=\deg(U_{h}\otimes E_{k,l})
\]
so we get that $\psi$ is a graded homomorphism, and from the simplicity
of $A$, it is a graded embedding.
\end{proof}
We can now use again the structure theorem for finite $G$-simple
algebras to conclude
\begin{cor}
Let $G$ be an arbitrary group, $B=M_{\bar{s}}(\KK)$ where $\bar{s}\in G^{r}$
and $A$ a finite dimensional $G$-simple algebra then there is a
graded embedding $A\hookrightarrow B$ iff $Id_{G}(B)\subseteq Id_{G}(A)$.
\end{cor}
\newpage{}

\section{Graded Embeddings of a Semisimple Algebra in a Finite Dimensional
Algebra}

In the previous section we showed that if $A,B$ are finite dimensional
$G$-simple algebras for $G$ abelian then there is a graded embedding
$A\hookrightarrow B$ iff $Id_{G}(B)\subseteq Id_{G}(A)$. 

We would first like to show that we cannot remove the simplicity of
$A$. One easy counter example is this - for any finite dimensional
$G$-graded algebra $B$ (simple or not), let $A=B\oplus B$ with
the grading $A_{g}=B_{g}\oplus B_{g}$ then $Id_{G}(A)=Id_{G}(B)\cap Id_{G}(B)=Id_{G}(B)$
and in particular we have $Id_{G}(B)\subseteq Id_{G}(A)$. On the
other hand there is no graded embedding $A\hookrightarrow B$ from
the simple fact that $\dim(A)=2\dim(B)>\dim(B)$. 

One might think that the theorem fails because $A$ contains two copies
of the same object, but even if $A$ is a direct sum of two non-isomorphic
simple algebras the theorem fails
\begin{example}
Let $B=M_{n+2}(\KK)$ be graded elementary by the group $G=\nicefrac{\ZZ}{10\ZZ}$
with the tuple $(0,1,1,....,1,3)$, then $B$ is $G$-simple. Let
$A_{1},A_{2}$ be two sub algebras of $B$ correspond to the sub tuple
$(0,1,1,...,1)$ and the sub tuple $(1,1,...,1,3)$ then we have $Id_{G}(B)\subseteq Id_{G}(A_{i})$
for $i=1,2$. $A_{i}$ are also $G$-simple and 
\[
Id_{G}(B)\subseteq Id_{G}(A_{1})\cap Id_{G}(A_{2})=Id_{G}(A_{1}\oplus A_{2})
\]
Let $A=A_{1}\oplus A_{2}$ and notice that $Id_{G}(A_{1})\neq Id_{G}(A_{2})$,
for example because $Supp_{G}(A_{1})=\left\{ 0,1,9\right\} $ while
$Supp_{G}(A_{2})=\left\{ 0,2,8\right\} $. 
\[
\dim(A_{1}\oplus A_{2})=\dim(A_{1})+\dim(A_{2})=2(1+n)^{2}>\left(2+n\right)^{2}=\dim(A)
\]
for large $n$, so it is not possible that $A=A_{1}\oplus A_{2}\hookrightarrow B$.\\

\end{example}
Notice that though in the last example we cannot embed $A$, which
is semisimple, in $B$, we can embed it in $B\oplus B$. One of the
results of this section is that this is always the case. 

We now remark that unless otherwise mentioned, there is no condition
on the field $\KK$ or the group $G$.\\

We first give an important property that simple algebras enjoy.

Suppose that $a,b\in M_{n}(\KK)$ then we can write 
\[
a=\sum a_{i,j}E_{i,j}\qquad b=\sum b_{i,j}E_{i,j}\qquad a_{i,j},b_{i,j}\in\KK
\]
If $a,b$ are not zero, then there are $i_{1},j_{1},i_{2},j_{2}\in\left\{ 1,...,n\right\} $
such that $a_{i_{1},j_{1}}\neq0$ and $b_{i_{2},j_{2}}\neq0$ so
\[
E_{1,i_{1}}\cdot a\cdot E_{j_{1},i_{2}}\cdot b\cdot E_{j_{2},1}=a_{i_{1},j_{1}}\cdot b_{i_{2},j_{2}}\cdot E_{1,1}\neq0\quad\Rightarrow\quad aE_{j_{1},i_{2}}b\neq0
\]
This condition, that if $a,b\neq0$ then we can find $x$ such that
$axb\neq0$ is true for any simple graded algebra and we state it
here.
\begin{lem}
\label{lem:product in simple algebras}Let $G$ be any group and $\KK$
any field. Let $A$ be a $G$-simple $\KK$ algebra
\begin{enumerate}
\item Suppose that $a,b\in A$ are non-zero homogeneous elements, then we
can find $x\in A$ homogeneous such that $axb\neq0$.
\item Let $f_{1},f_{2}\notin Id_{G}(A)$ homogeneous polynomials then there
is a homogeneous polynomial $h$ such that $f_{1}\cdot h\cdot f_{2}\notin Id_{G}(A)$.
\end{enumerate}
\end{lem}
\begin{proof}

\begin{enumerate}
\item Notice that $I=\left\{ a\in A\;\mid\; aA=0\right\} $ is a graded
ideal in $A$. Since $A^{2}\neq0$ then $I\neq A$, and from simplicity
$I=0$, or in other words, if $a\neq0$ then $aA\neq0$. For a nonzero
homogeneous element $a\in A$, the set $J=\left\{ b\in A\;\mid\; aAb=0\right\} $
is again a graded ideal in $A$. Find $0\neq a'\in aA$, so $a'A\neq0$,
and therefore $J\neq A$, and from simplicity $J=0$. If $a,b\neq0$
are homogeneous then $aAb\neq0$, and in particular we can find homogeneous
$x\in A$ such that $axb\neq0$.
\item Obvious from the first part of the lemma.
\end{enumerate}
\end{proof}
The algebras in this section are direct sums of simple algebra, so
we would like to see what it means to have an ideal of one such algebra
inside another. We recall again that $Id_{G}(\bigoplus_{1}^{n}B_{i})=\bigcap_{1}^{n}Id_{G}(B_{i})$
for any algebras $B_{i}$. In general, if $\bigcap_{1}^{n}Id_{G}(B_{i})\subseteq Id_{G}(A)$,
then it doesn't mean that there is some $i$ such that $Id_{G}(B_{i})\subseteq Id_{G}(A)$.
The next lemma shows that this is true, if we also assume that $A$
is simple.
\begin{lem}
\label{lem:intersection of algebras in simple}Suppose that $B_{1},...,B_{n}$
are $G$-graded algebras and $A$ is a $G$-simple algebra such that
$Id_{G}(\bigoplus B_{i})=\bigcap Id_{G}(B_{i})\subseteq Id_{G}(A)$,
then there is $i\in\left[n\right]$ such that $Id_{G}(B_{i})\subseteq Id_{G}(A)$.\end{lem}
\begin{proof}
Suppose that the claim is false, then for each $i$ there is a homogeneous
polynomial $f_{i}\in Id_{G}(B_{i})\backslash Id_{G}(A)$, but then
from \prettyref{lem:product in simple algebras} we could find $h_{i}\in\FF\left\langle X_{G}\right\rangle $
such that $f=\prod\left(f_{i}\cdot h_{i}\right)\notin Id_{G}(A)$.
On the other hand $f\in Id_{G}(B_{i})$ for each $i$ so $f\in\bigcap Id_{G}(B_{i})\subseteq Id_{G}(A)$
and we get a contradiction.
\end{proof}
We now have a semisimple algebra $A=A_{1}\oplus\cdots\oplus A_{n}$
with $Id_{G}(A)=\bigcap Id_{G}(A_{i})$ so we might remove some of
the algebras $A_{i}$ without increasing the ideal of identities.
If $A_{1},A_{2}$ are $G$-graded and $Id_{G}(A_{1})\subseteq Id_{G}(A_{2})$,
then the ideal of identities cannot separate between $A_{1}\oplus A_{2}$
and $A_{1}$ because $Id_{G}(A_{1}\oplus A_{2})=Id_{G}(A_{1})\cap Id_{G}(A_{2})=Id_{G}(A_{1})$.
To avoid such cases we give the next definition
\begin{defn}
Let $A_{1},...,A_{n}$ be distinct $G$-graded algebras. We call the
set $\left\{ A_{i}\mid1\leq i\leq n\right\} $ minimal with respect
to identities (or just minimal) if for each $1\leq j\leq n$ we have
${\displaystyle \bigcap_{k=1}^{n}}Id_{G}(A_{k})\subsetneq{\displaystyle \bigcap_{k\neq j}}Id_{G}(A_{k})$.
\end{defn}
If there are $A_{i},A_{j}$ in the set $S$ of graded algebras such
that $Id_{G}(A_{i})\subseteq Id_{G}(A_{j})$ then $S$ is not minimal.
The other direction is not true in general, but if the algebras are
simple then the last lemma shows that it is true. If $S=\left\{ A_{1},...,A_{n}\right\} $
is a set of $G$-simple algebras that is not minimal, then wlog $\bigcap_{2}^{n}Id_{G}(A_{i})=\bigcap_{1}^{n}Id_{G}(A_{j})$
so $\bigcap_{2}^{n}Id_{G}(A_{i})\subseteq Id_{G}(A_{1})$, but then
\prettyref{lem:intersection of algebras in simple} shows that there
is $i\in\left\{ 2,...,n\right\} $ such that $Id_{G}(A_{i})\subseteq Id_{G}(A_{1})$.
In other words, the set $S$ is minimal iff the ideals $Id_{G}(A_{i})$
are all distinct and there are no inclusions $Id_{G}(A_{i})\subseteq Id_{G}(A_{j})$
for $i\neq j$.

We can now take a condition such as $Id_{G}(\bigoplus B_{j})\subseteq Id_{G}(\bigoplus A_{i})$
and decompose it into {}``smaller'' conditions about each $A_{i}$
separately.
\begin{lem}
Let $A_{1},...,A_{n}$ and $B_{1},...,B_{m}$ be $G$-simple algebras
and define $A=\bigoplus_{1}^{n}A_{i},\quad B=\bigoplus_{1}^{m}B_{j}$ 
\begin{enumerate}
\item If $Id_{G}(B)\subseteq Id_{G}(A)$ then there is a function $\tau:\left[n\right]\rightarrow\left[m\right]$
such that $Id_{G}(B_{\tau(i)})\subseteq Id_{G}(A_{i})$.
\item If $Id_{G}(A)=Id_{G}(B)$ and the sets $\left\{ A_{i}\right\} _{1}^{n}$
and $\left\{ B_{j}\right\} _{1}^{m}$ are minimal, then $m=n$ and
there is a permutation $\tau:\left[n\right]\rightarrow\left[n\right]$
such that $Id_{G}(A_{i})=Id_{G}(B_{\tau(i)})$
\end{enumerate}
\end{lem}
\begin{proof}

\begin{enumerate}
\item Let $i\in\left[n\right]$ be fixed then 
\[
\bigcap_{1}^{m}Id_{G}(B_{j})=Id_{G}(B)\subseteq Id_{G}(A)=\bigcap_{1}^{n}Id_{G}(A_{k})\subseteq Id(A_{i})
\]
We can now use \ref{lem:intersection of algebras in simple} to find
some $j\in\left[m\right]$ such that $Id_{G}(B_{j})\subseteq Id(A_{i})$
- denote this $j$ by $j=\tau(i)$. 
\item From part $(1)$ we can find functions $\tau:\left[n\right]\rightarrow\left[m\right]$
and $\pi:\left[m\right]\rightarrow\left[n\right]$ such that $Id_{G}(B_{\tau(i)})\subseteq Id_{G}(A_{i})$
and $Id_{G}(A_{\pi(j)})\subseteq Id_{G}(B_{j})$, so for all $i$
\[
Id_{G}(A_{\pi(\tau(i))})\subseteq Id_{G}(B_{\tau(i)})\subseteq Id_{G}(A_{i})
\]
but the set $\left\{ A_{i}\right\} _{1}^{n}$ is minimal and the $A_{i}$
are simple, so from the remark before the lemma we must have $\pi(\tau(i))=i$
and in particular $Id_{G}(B_{\tau(i)})=Id_{G}(A_{i})$, and $\tau$
is injective. The same is true in the other direction, so $\pi$ is
injective and therefore $n=m$, and $\tau$ is a permutation on $\left[n\right]$
such that $Id_{G}(A_{i})=Id_{G}(B_{\tau(i)})$.
\end{enumerate}
\end{proof}
The function $\tau$ in part $(1)$ is not injective in general, but
if we duplicate $B$ enough times we can find a new function $\tilde{\tau}$
from the simple components of $A$ to the simple components of $B^{N}$,
for $N$ large enough, such that $\tilde{\tau}$ is injective. If
$Id_{G}(B_{j})\subseteq Id_{G}(A_{i})$ for $B_{j},A_{i}$ simple
means that there is a graded embedding $A_{i}\hookrightarrow_{G}B_{j}$
then we can find a graded embedding $A\hookrightarrow_{G}B^{N}$ for
$N$ large enough, so we just proved that
\begin{thm}
Let $A_{1},...,A_{n}$ and $B_{1},...,B_{m}$ be $G$-simple algebras
and define $A=\bigoplus_{1}^{n}A_{i},\quad B=\bigoplus_{1}^{m}B_{j}$ 
\begin{enumerate}
\item Suppose that whenever $Id_{G}(B_{j})\subseteq Id_{G}(A_{i})$ than
there is a graded embedding $A_{i}\hookrightarrow_{G}B_{j}$. If $Id_{G}(B)\subseteq Id_{G}(A)$
then there is a graded embedding $A\hookrightarrow_{G}B^{N}$ for
$N$ large enough.
\item Suppose that whenever $Id_{G}(B_{j})=Id_{G}(A_{i})$ then there is
a graded isomorphism $A_{i}\cong_{G}B_{j}$. If the sets $\left\{ A_{i}\right\} _{1}^{n}$
and $\left\{ B_{j}\right\} _{1}^{m}$ are minimal and $Id_{G}(A)=Id_{G}(B)$
then $m=n$, and there is a permutation $\tau\in S_{n}$ such that
$A_{i}\cong_{G}B_{\tau(i)}$, and in particular $A\cong_{G}B$.
\end{enumerate}
\end{thm}
For example, part $(1)$ is true for $G$ abelian, $A_{i},B_{j}$
are finite dimensional and $\KK$ is algebraically closed with $char(\KK)=0$.
Part $(2)$ is also true when the conditions in one of \cite{aljadeff_simple_2011,koshlukov_identities_2010}
is true. Notice that by our construction, if $A$ is simple then we
can always choose $N=1$ in part $(1)$.

We continue to the next generalization where $B$ can be any finite
dimensional algebra, and we want to reduce our problem to the semisimple
case. 

We first recall that in the non-graded case, if $B$ is Artinian (and
in particular if $B$ is finite dimensional) then its Jacobson radical
$J$ is nilpotent and if $J=0$ then $B$ is a direct sum of simple
algebras (\cite{herstein_noncommutative_2005}). 

Let $J$ be the Jacobson radical of $B$ and $\pi:B\rightarrow\nicefrac{B}{J}$
the projection map, then $B_{ss}=\nicefrac{B}{J}$ is semi simple,
and therefore a direct sum of simple algebras. In general we only
know that $Id(B)\subseteq Id(B_{ss})$ so if $Id(B)\subseteq Id(A)$,
then we don't know if there is any relation between $Id(A)$ and $Id(B_{ss})$.
Assume that we can show that we actually have $Id(B)\subseteq Id(B_{ss})\subseteq Id(A)$
so in a sense we don't add too much identities when we move from $Id(B)$
to $Id(B_{ss})$.

Wedderburn Malcev theorem states that we can write $B=B_{ss}\oplus J$
where the direct sum is as vector spaces, so the idea is to show that
\[
Id(B)\subseteq Id(A)\;\Rightarrow\; Id(B_{ss})\subseteq Id(A)\;\Rightarrow\; A\hookrightarrow B_{ss}\;\Rightarrow\; A\hookrightarrow B
\]

A similar process can be done in the graded case. Define $J_{G}$
to be the graded Jacobson radical of $B$, so $J_{G}$ is the intersection
of all maximal right graded ideals of $B$. Assume now that $G$ is
a finite group and that $\left|G\right|^{-1}\in B$. In \cite{cohen_group-graded_1984}
(theorem 4.4), Cohen and Montgomery proved that this condition implies
that the graded Jacobson radical $J_{G}$ equals to the non-graded
radical $J$. The result about the Wedderburn Malcev decomposition
$B\cong B_{ss}\oplus J_{G}$ can be found in \cite{stefan_wedderburn-malcev_1999}
(corollary 2.8). It is also well known that if $J_{G}=\left\{ 0\right\} $
then $B$ is a finite direct product of $G$-simple algebras.

Before we continue, we remark that some of these properties are true
in a more general setting.
\begin{lem}
\label{lem:Upper bound for A_ss}Assume that $G$ is a finite group.
Let $A$ be $G$-semisimple algebra, and $B$ a finite dimensional
$G$-graded algebra such that $Id_{G}(B)\subseteq Id_{G}(A)$ and
$\left|G\right|^{-1}\in B$. Denote by $J_{G}$ the graded Jacobson
radical of $B$ and $B_{ss}=\nicefrac{B}{J_{G}}$ its graded semi
simple image, then $Id_{G}(B_{ss})\subseteq Id_{G}(A)$. \end{lem}
\begin{proof}
Assume by negation that there is $f\in Id_{G}(B_{ss})\backslash Id_{G}(A)$,
wlog homogeneous. $B$ is finite dimensional so $J$ is nilpotent,
and because $\left|G\right|^{-1}\in B$ then $J_{G}=J$ and we can
find $k\in\NN$ such that $\left(J_{G}\right)^{k}=0$. $Id_{G}(B_{ss})$
is an ideal so $h_{i}\cdot f\in Id_{G}(B_{ss})$ for any homogeneous
polynomial $h_{i}$. Since $B_{ss}\cong\nicefrac{B}{J_{G}}$ then
$h_{i}\cdot f$ sends any graded assignment in $B$ to $J_{G}$.

From \prettyref{lem:product in simple algebras}, since $f\notin Id_{G}(A)$
and $A$ is $G$-simple, we can find $h_{i}$ such that $\prod_{1}^{k}\left(h_{i}\cdot f\right)\notin Id_{G}(A)$,
but for any assignment $\bar{b}$ in $B$ we get that
\[
\prod_{1}^{k}\left(h_{i}\cdot f\right)(\bar{b})\in\prod_{1}^{k}J_{G}=0\quad\Rightarrow\quad\prod_{1}^{k}\left(h_{i}\cdot f\right)\in Id_{G}(B)
\]
 and we get a contradiction since $Id_{G}(B)\subseteq Id_{G}(A)$.

This shows that there are no $f$ such that $f\in Id_{G}(B_{ss})\backslash Id_{G}(A)$,
or in other words $Id_{G}(B_{ss})\subseteq Id_{G}(A)$.
\end{proof}
Now combine this lemma, the last theorem and Wedderburn-Malcev theorem
for the graded case to get
\begin{thm}
Let $G$ be a finite abelian group. Let $A,B$ be two be $G$-graded
algebras over algebraically closed field $\KK$ with $char(\KK)=0$.
Assume that $A$ is $G$-semisimple, $B$ has a a unit and $Id_{G}(B)\subseteq Id_{G}(A)$.
Then there is a graded embedding $A\hookrightarrow_{G}B^{N}$ for
$N$ large enough, and if $A$ is simple then we can choose $N=1$.
\end{thm}
\bibliographystyle{unsrt}
\addcontentsline{toc}{section}{\refname}\bibliography{graded_identities}

\end{document}